\documentclass{amsart}

\usepackage{a4wide}
\usepackage{amssymb}
\usepackage[utf8]{inputenc}
\usepackage{graphicx}
\usepackage{color}
\usepackage[hyperref]{hyperref}

\renewcommand{\phi}{\varphi}
\newcommand{\C}{{\mathbb{C}}}
\newcommand{\HyperbolicPlane}{{\mathbb{H}}}
\newcommand{\N}{{\mathbb{N}}}
\newcommand{\R}{{\mathbb{R}}}

\newcommand{\Z}{{\mathbb{Z}}}
\newcommand{\1}{{\mathbf{1}}}
\renewcommand{\epsilon}{\varepsilon}
\renewcommand{\theta}{\vartheta}
\renewcommand{\S}{{\mathbb{S}}}
\newcommand{\T}{{\mathbb{T}}}

\newcommand{\Disk}[1][{}]{{\mathbb{D}_{#1}}}
\newcommand{\Reeb}{{X_\mathrm{Reeb}}}
\newcommand{\solutions}{{\widetilde{\mathcal{M}}}}
\newcommand{\moduli}{{\mathcal{M}}}
\newcommand{\z}{{\mathbf{z}}}
\newcommand{\lie}[1]{{\mathcal{L}_{#1}}}
\newcommand{\pairing}[2]{{\langle{#1}|{#2}\rangle}}
\newcommand{\plastikstufe}[1]{{\mathcal{PS}({#1})}}
\newcommand{\GPS}{{\mathrm{GPS}}}
\newcommand{\norm}[1]{{\lVert #1\rVert}}
\newcommand{\abs}[1]{{\left\lvert #1\right\rvert}}
\newcommand{\selfintersection}[1]{{{#1}_{\between}}}
\newcommand{\overtwisted}{{\mathbb{D}_\mathrm{OT}}}

\newcommand{\lcan}{{\lambda_{\mathrm{can}}}}
\newcommand{\POS}{{\mathbf{q}}}
\newcommand{\MOM}{{\mathbf{p}}}

\DeclareMathOperator{\End}{End}
\DeclareMathOperator{\evaluationmap}{ev}

\DeclareMathOperator{\id}{id}
\DeclareMathOperator{\ImaginaryPart}{Im}

\DeclareMathOperator{\RealPart}{Re}
\DeclareMathOperator{\SU}{SU}

\theoremstyle{plain}
\newtheorem{theorem}{Theorem}
\newtheorem{proposition}[theorem]{Proposition}
\newtheorem{corollary}[theorem]{Corollary}
\newtheorem{question}{Question}

\theoremstyle{remark}
\newtheorem{remark}{Remark}
\newtheorem{example}{Example}

\theoremstyle{definition}
\newtheorem*{defi}{Definition}

\begin{document}

\title[Size of neighborhoods in contact topology]{Some remarks on the
  size of tubular neighborhoods in contact topology and fillability}

\author[K.\ Niederkrüger]{Klaus Niederkrüger}

\email[K.\ Niederkrüger]{kniederk@umpa.ens-lyon.fr}

\address[K.\ Niederkrüger]{École Normale Supérieure de Lyon\\
  Unité de Mathématiques Pures et Appliquées\\
  UMR CNRS 5669\\ 46 allée d'Italie\\ 69364 LYON Cedex 07\\
  France}

\author[F.\ Presas]{Francisco Presas}

\email[F.\ Presas]{fpresas@imaff.cfmac.csic.es}

\address[F.\ Presas]{Departamento de Álgebra \\
  Facultad de Matemáticas \\ Universidad Complutense de Madrid \\
  Plaza de Ciencias n\textordmasculine~3 \\ 28040 Madrid \\ Spain}

\begin{abstract}
  The well-known tubular neighborhood theorem for contact submanifolds
  states that a small enough neighborhood of such a submanifold $N$ is
  uniquely determined by the contact structure on $N$, and the
  conformally symplectic structure of the normal bundle.  In
  particular, if the submanifold $N$ has trivial normal bundle then
  its tubular neighborhood will be contactomorphic to a neighborhood
  of $N\times\{\mathbf{0}\}$ in the model space $N\times \R^{2k}$.

  In this article we make the observation that if $(N,\xi_N)$ is a
  $3$--dimensional overtwisted submanifold with trivial normal bundle
  in $(M,\xi)$, and if its model neighborhood is sufficiently large,
  then $(M,\xi)$ does not admit an exact symplectic filling.
\end{abstract}

\maketitle

In symplectic geometry many invariants are known that measure in some
way the ``size'' of a symplectic manifold.  The most obvious one is
the total volume, but this is usually discarded, because one can
change the volume (in case it is finite) by rescaling the symplectic
form without changing any other fudamental property of the manifold.
The first non-trivial example of an invariant based on size is the
symplectic capacity \cite{Gromov_Kurven}.  It relies on the fact that
the size of a symplectic ball that can be embedded into a symplectic
manifold does not only depend on its total volume but also on the
volume of its intersection with the symplectic $2$--planes.

Contact geometry does not give a direct generalization of these
invariants.  The main difficulties stem from the fact that one is only
interested in the contact structure, and not in the contact form, so
that the total volume is not defined, and to make matters worse the
whole Euclidean space $\R^{2n+1}$ with the standard structure can be
compressed by a contactomorphism into an arbitrarily small open ball
in $\R^{2n+1}$.

A more successful approach consists in studying the size of the
neighborhood of submanifolds.  This can be considered to be a
generalization of the initial idea since contact balls are just
neighborhoods of points.  In the literature this idea has been pursued
by looking at the tubular neighborhoods of circles.  Let
$(N,\alpha_N)$ be a closed contact manifold.  The product $N\times
\R^{2k}$ carries a contact structure given as the kernel of the form
$\alpha_N + \sum_{j=1}^k (x_j\,dy_j - y_j\,dx_j)$, where
$(x_1,\dotsc,x_k,y_1,\dotsc,y_k)$ are the coordinates of the Euclidean
space.  If $(N,\alpha_N)$ is a contact submanifold of a manifold $(M,
\alpha)$ that has trivial (as conformal symplectic) normal bundle,
then one knows by the tubular neighborhood theorem that $N$ has a
small neighborhood in $M$ that is contactomorphic to a small
neighborhood of $N\times\{\mathbf{0}\}$ in the product space $N\times
\R^{2k}$.

The contact structure on a solid torus $V$ in $\S^1\times\R^2$ depends
in an intricate way on the radius of $V$ \cite{EliashbergSolidTori}.
Later, examples of transverse knots in $3$--manifolds were found whose
maximal neighborhood is only contactomorphic to a small disk bundle in
$\S^1\times \R^2$ \cite{Cabling}.  This is proved by measuring the
slope of the characteristic foliation on the boundary of cylinders.

A different approach has been taken in \cite{NonContactSqueezing}.
There it has been shown that a solid torus around
$\S^1\times\{\mathbf{0}\}$ in $\S^1\times\R^{2k}$ of radius $R$ cannot
be ``squeezed'' into a solid torus of radius $r$, if $k\ge 2$, $R>2$
and $r<2$.  However note that squeezing in the context of
\cite{NonContactSqueezing} is different from the expected definition,
and refers to the question of whether one subset of a contact manifold
can be deformed by a global isotopy into another one.

The observation on which the present article is based is that
sufficiently large neighborhoods of $N\times\{\mathbf{0}\}$ in
$N\times \R^{2k}$ contain a generalized plastikstufe (for a definition
of the GPS see Section~\ref{sec: definition GPS}), if $N$ is an
overtwisted $3$--manifold.  The construction of a GPS in a tubular
neighborhood is explained in Section~\ref{sec: construction of GPS in
  tubular nbhd}.  In Section~\ref{sec: proof of nonfillability}, we
show that the existence of a GPS implies nonfillability, and so it
follows in particular that an overtwisted contact manifold that is
embedded into a fillable manifold cannot have a ``large''
neighborhood.

Unfortunately, the definition of ``large'' is rather subtle and does
not lead to a numerical invariant, because such an invariant would
depend on the contact form on the submanifold.  One could simply
multiply any contact form $\alpha_N + \sum_{j=1}^k (x_j\,dy_j -
y_j\,dx_j)$ by a constant $\lambda > 0$, and then rescale the radii in
the plane by a transformation $r_j \mapsto r_j/\sqrt{\lambda}$ to
change the numerical invariant.

\subsection*{Acknowledgments}
K.\ Niederkrüger works at the \emph{ENS de Lyon} funded by the project
\emph{Symplexe} of the \emph{Agence Nationale de la Recherche} (ANR).

Several people helped us writing this article: Many ideas and examples
are due to Emmanuel Giroux (in particular Example~\ref{example:
  overtwisted in fillable manifold}).  Example~\ref{example: branched
  cover} was found in discussions with Hansjörg Geiges.  We thank Ana
Rechtman Bulajich for helping us clarifying the general ideas of the
paper, and Mohammed Abouzaid and Pierre Py for extremely valuable
discussions about holomorphic curves.  Paolo Ghiggini pointed out to
us known results about neighborhoods of transverse knots.

\section{Examples}

First we give an easy example that shows that embedding an overtwisted
$3$--manifold into a fillable contact manifold does not pose a
fundamental problem in positive codimension.

\begin{example}\label{example: overtwisted in fillable manifold}
  Let $M$ be an arbitrary orientable closed $3$--manifold.  Its unit
  cotangent bundle $\S\bigl(T^* M\bigr) \cong M\times \S^2$ has a
  contact structure defined by the canonical $1$--form.  The cotangent
  bundle $T^*M$ together with the form $d\lcan$ is an exact symplectic
  filling (and in fact, it can even be turned into a Stein filling).

  Any contact manifold $(M,\alpha)$ can be embedded into
  $\bigl(\S\bigl(T^* M\bigr), \lcan\bigr)$ just by normalizing
  $\alpha$ so that $\norm{\alpha} = 1$.  This defines a section in
  $\sigma:\, M \to \S\bigl(T^* M\bigr)$ with $\sigma^* \lcan =
  \alpha$.  This means that every (and in particular also every
  overtwisted overtwisted one) contact $3$--manifold can be embedded
  into a Stein fillable contact $5$--manifold.
\end{example}

Embedding a contact $3$--manifold into a contact manifold of dimension
$7$ or higher restricts by using the $h$--principle and a general
position argument to a purely topological question.

Our second and third example show that contact submanifolds can have
infinitely large tubular neighborhoods.

\begin{example}
  Let $(M,\alpha)$ be an arbitrary contact manifold, and let
  $(\S^{2n-1}, \xi_0)$ be the standard sphere.  If $\dim M \ge 2n-1$,
  then it is easy to give a contact embedding
  \begin{equation*}
    \Bigl(\S^{2n-1}\times \R^{2k}, \alpha_0
    + \sum_{j=1}^k (x_j\,dy_j -  y_j\,dx_j)\Bigr) \hookrightarrow
    \bigl(M, \alpha\bigr) \;.
  \end{equation*}
  The proof works in two steps.  For the embedding
  \begin{equation*}
    \Bigl(\S^{2n-1}\times \R^{2k}, \alpha_0 +
    \sum_{j=1}^k (x_j\,dy_j -  y_j\,dx_j)\Bigr) \hookrightarrow
    \bigl(\S^{2n+2k - 1}, \alpha_0\bigr)
  \end{equation*}
  simply use the map $(z_1,\dotsc,z_n;x_1,y_1,\dotsc,x_k,y_k) \mapsto
  \frac{1}{\sqrt{1+\norm{\mathbf{x}}^2 + \norm{\mathbf{y}}^2}} \,
  (z_1,\dotsc, z_n,x_1+iy_1,\dotsc,x_k+iy_k)$.  Since $(\S^{2N - 1},
  \alpha_0)$ with one point removed is contactomorphic to $\R^{2N-1}$
  with standard contact structure (see for example
  \cite[Proposition~2.13]{Geiges_Handbook}) and since it is possible
  to embed the whole $\R^{2N-1}$ into an arbitrary small Darboux chart
  (see for example \cite[Proposition~3.1]{ContactBallCoverings}), it
  follows that a general $(M,\alpha)$ contains embeddings of
  $\bigl(\S^{2n-1}\times \R^{2k}, \alpha_0 + \sum_{j=1}^k (x_j\,dy_j -
  y_j\,dx_j)\bigr)$.
\end{example}

\begin{example}
  A generalization is obtained by choosing a contact manifold
  $(N,\alpha_N)$ that has an exact symplectic filling $(W,\omega =
  d\lambda)$.  The Lioville field $X_L$ is globally defined (see
  Section~\ref{sec: definition fillability}), and we can use its
  negative flow for finding an embedding of the lower half of the
  symplectization $(-\infty,0]\times N$ where $\lambda$ pulls back to
  $e^t\,\alpha_N$.  The manifold $\S^1\times W$ is together with the
  $1$--form $d\theta + \lambda$ a contact manifold.

  The standard model $(N\times\R^2,\alpha_N + r^2\,d\phi)$ can be
  glued outside the $0$--section onto $\S^1\times W$, and this
  construction yields a closed contact manifold that contains the
  embedding of $N\times \R^2$.  This example can be seen as an open
  book with binding $N$, page $W$, and trivial monodromy.
\end{example}

Not much is known about the different contact structures on
$\R^{2n+1}$ for $n\ge 2$.  There exists the standard contact structure
$\xi_0$, and many different constructions to produce structures that
are not isomorphic to the standard one (for example
\cite{BatesPeschke, MullerExoticR6, NiederkruegerPlastikstufe}).
Unfortunately we do not have the techniques to decide whether these
exotic contact structures are different from each other.  A contact
structure $\xi$ on $\R^{2n+1}$ is called \textbf{standard at infinity}
\cite{EliashbergContactR3}, if there exists a compact subset $K$ of
$\R^{2n+1}$ such that $\bigl(\R^{2n+1}-K, \xi\bigr)$ is
contactomorphic to $\bigl(\R^{2n+1}-\Disk[R],\xi_0\bigr)$ for a closed
disk of an arbitrary radius $R$.  A contact structure $\xi$ on
$\R^{2n+1}$ only admits a one-point compactificaton to a contact
structure on the sphere, if $\xi$ is standard at infinity.  For most
exotic contact structures it is not known whether they are standard at
infinity or not.  The only exception known to us so far was given in
\cite{PSOvertwistedEverywhere}, where by removing one point from the
sphere, we obtained an exotic contact structure $\xi_{PS}$ on
$\R^{2n+1}$ that \emph{is} standard at infinity (but see also
Example~\ref{example: branched cover}).  A rather degenerate way of
producing a contact structure that is not standard at infinity
consists in taking the standard structure on $\R^{2n+1}$, and do the
connected sum at every point $(0,\dotsc,0,k)\in \R^{2n+1}$ with
$k\in\Z$ with the sphere $(\S^{2n+1}, \xi_{PS})$.

Corollary~\ref{GPS in manifolds with overtwisted blah} below yields a
very explicit way to construct an exotic contact structure that is not
standard at infinity.

\begin{example}
  The contact manifold
  \begin{equation*}
    \bigl(\R^3\times \C^k, \alpha_- + \sum_{j=1}^k r_j^2\,d\theta_j\bigr) \;,
  \end{equation*}
  where $(r_j,\theta_j)$ are polar coordinates on the $j$--th factor
  of $\C^k$, does not embed into the standard sphere, and is hence not
  contactomorphic to the standard contact structure on $\R^{2k+3}$.
  Let $K\subset \R^3\times\C^k$ be an arbitrary compact subset.  By
  the same argument, it is easy to see that $\bigl(\R^3\times \C^k -
  K, \alpha_- + \sum_j r_j^2\,d\theta_j\bigr)$ still contains a GPS,
  so in particular it cannot be embedded into a ``punctured'' set $U
  -\{p\}\subset \bigl(\R^{2k+3},\alpha_0\bigr)$ with the standard
  contact structure.  It follows that $\bigl(\R^3\times \C^k, \alpha_-
  + \sum_j r_j^2\,d\theta_j\bigr)$ is ``non standard at infinity''.
\end{example}

Let $(M,\alpha)$ be a closed contact manifold that contains a contact
submanifold $N$ of codimension~$2$ with trivial normal bundle.  A
$k$--fold \textbf{contact branched covering} over $M$ consists of a
closed manifold $\widetilde M$, and a smooth surjective map $f:\,
\widetilde M \to M$ such that the map $f$ is a smooth $k$--fold
covering over $M-N$, and there is an open neighborhood $\widetilde
U\subset \widetilde M$ of $f^{-1}(N)$ diffeomorphic to $N\times
\Disk_\epsilon$, and a neighborhood $U\subset M$ of $N$ diffeomorphic
to $N\times \Disk_{\epsilon^k}$ such that the map $f$ takes the form
\begin{equation*}
  f:\, N\times \Disk_\epsilon \to N\times \Disk_{\epsilon^k},\,
  (p,z) \mapsto (p,z^k)\;,
\end{equation*}
when restricted to $\widetilde U$  (see \cite{GonzaloBranchedCover}).

Using the branched covering, it is easy to define a contact structure
on $\widetilde M$.  First isotope $\alpha$ in such a way that it
becomes $\left.\alpha\right|_{TN} + r^2\,d\phi$ on a subset $N\times
\Disk_\delta \subset N\times \Disk_{\epsilon^k}$ for some $\delta >
0$.  The pull-back $\widetilde \alpha := f^*\alpha$ defines on
$\widetilde M$ a $1$--form that satisfies away from $f^{-1}(N)$
everywhere the contact property.  Over the branching locus
$f^{-1}(N)$, there is a subset $N\times \Disk_{\sqrt[k]{\delta}}$ in
$\widetilde U$ where $\widetilde \alpha$ evaluates to
$\left.\alpha\right|_{TN} + kr^{2k}\,d\phi$.

Remove the fiber $N\times \{0\}$ from $\widetilde U$ and glue in
$N\times \Disk_{\delta}$ via the map $F:\,(p,re^{i\phi}) \mapsto
(p,\sqrt[k]{r}\, e^{i\phi})$ along $N\times
\bigl(\Disk_{\sqrt[k]{\delta}} - \{0\}\bigr)$.  The pull-back
$F^*\widetilde \alpha$ yields $\left.\alpha\right|_{TN} + kr^2\,d\phi$
on the punctured disk bundle, which we can easily extend to the whole
patch we are gluing in.  We denote this slightly modified contact form
again by $\widetilde\alpha$.  By using a linear stretch map on the
disk, we finally obtain that the submanifold $f^{-1}(N) \cong N$ has
with respect to the model form
\begin{equation*}
  \bigl(N\times \R^2, \left.\alpha\right|_{TN} + r^2\,d\phi \bigr)
\end{equation*}
a neighborhood inside $(\widetilde M, \widetilde \alpha)$ that is at
least of size $\sqrt{k}\,\delta$.

\begin{example}\label{example: branched cover}
  There is an interesting contact structure on the odd dimensional
  spheres $\S^{2n-1}\subset \C^n$ given as the kernel of the
  $1$--forms
  \begin{equation*}
    \alpha_- = i\,\sum_{j=1}^n \bigl(z_j\, d\bar z_j - \bar z_j\,dz_j\bigr)
    - i \, \bigl(f\,d\bar f - \bar f \, df\bigr)\,
  \end{equation*}
  with $f(z_1,\dotsc,z_n) = z_1^2 + \dotsm + z_n^2$.  This form is
  compatible with the open book with binding $B=f^{-1}(0)$, and
  fibration map $\theta = \bar f/\abs{f}$.  In abstract terms, this is
  the open book with page $P \cong T^*\S^{n-1}$ and monodromy map
  corresponding to the negative Dehn-Seidel twist.

  An interesting feature of these spheres is that they can be stacked
  into each other via the natural inclusions $\S^3 \hookrightarrow
  \S^5 \hookrightarrow \S^7\hookrightarrow \dotsm$ respecting the
  contact form, and that $(\S^3,\alpha_-)$ is overtwisted.

  We find a contact branched cover $f:\,\S^5 \to (\S^5,\alpha_-)$
  given by $f(z_1,z_2,z_3) = \frac{(z_1,z_2,z_3^k)}
  {\norm{(z_1,z_2,z_3^k)}}$ that is branched along $\S^3$.  By
  choosing $k$ large enough, we will obtain with the construction
  described above a contact structure on $\S^5$ that contains an
  embedding of $(\S^3,\alpha_-)$ with an arbitrary large neighborhood.
  According to Corollary~\ref{GPS in manifolds with overtwisted blah}
  and Theorem~\ref{theorem: main result}, this contact structure will
  not admit an exact symplectic filling.

  This result is unsatisfactory, since we do not get an explicit value
  for $k$.  In fact, we expect that $(\S^3,\alpha_-)$ already has a
  large neighborhood in any of the spheres $(\S^{2n-1},\alpha_-)$ so
  that taking $k=1$ (that means not taking any branched covering at
  all) should already be sufficient.
\end{example}

\newcommand{\thisgoesout}{
\begin{example}
  Consider $(\S^3,\alpha_0)$ embedded in its unit cotangent bundle as
  described above.  One explicit model of the unit cotangent bundle
  can be obtained by embedding the sphere $\S^{2n-1}$ into $\R^{2n}$
  where it satisfies the equation $\norm{\POS} = 1$ with $\POS =
  (q_1,\dotsc,q_{2n})$.  Using the standard metric, the cotangent
  bundle can be identified with the subset $\bigl\{(\POS,\MOM) \in
  \R^{2n}\times \R^{2n} \bigm|\,\norm{\POS} = 1, \pairing{\POS}{\MOM}
  = 0\bigr\}$, and finally the unit cotangent bundle
  $\S(T^*\S^{2n-1})$ is given by
  \begin{equation*}
    \bigl\{(\POS,\MOM) \in \R^{2n}\times \R^{2n} \bigm|\,
    \norm{\POS} = \norm{\MOM} = 1,\, \pairing{\POS}{\MOM} = 0\bigr\} \;.
  \end{equation*}
  The canonical contact form $\lcan = \MOM\,d\POS$ can also be written
  as $\frac{1}{2}\,(\MOM\,d\POS - \POS\,d\MOM)$.  The embedding of
  $\S^{2n-1}$ into its unit cotangent bundle given by the contact form
  $\alpha_0$ is
  \begin{equation*}
    F:\, (x_1,x_2,\dotsc,x_{2n-1},x_{2n}) \mapsto
    \bigl((x_1,x_2,\dotsc,x_{2n-1},x_{2n});
    (-x_2,x_1,\dotsc,-x_{2n},x_{2n-1})\bigr)\;.
  \end{equation*}

  We can introduce complex coordinates $z_j = q_j - i p_j$ such that
  $\R^{2n} \times \R^{2n}$ becomes $\C^{2n}$, the canonical contact
  structure becomes $-\frac{i}{4}\, (\z\,d\bar\z - \bar \z\, d\z)$,
  and the cotangent bundle can be conveniently described by the
  equations $\norm{\z}^2 = 2$, and $z_1^2 + \dotsm + z_{2n}^2 = 0$,
  because $z_1^2 + \dotsm + z_{2n}^2 = \norm{\POS}^2 - \norm{\MOM}^2 -
  2i\,\pairing{\POS}{\MOM}$.  The embedding $F$ is now
  \begin{equation*}
    F:\, (z_1,\dotsc,z_n) \mapsto (z_1, i z_1,\dotsc, z_n,i z_n)\;.
  \end{equation*}

  The linear map on $\C^{2n}$ given by blocks of the form
  $\frac{1}{\sqrt{2}}\,\begin{pmatrix} 1 & -i \\ -i & 1 \end{pmatrix}$
  lies in $\SU(n)$, and hence preserves the canonical contact
  structure.  Apply this transformation to $\C^{2n}$, then the
  equations that define the unit cotangent bundle become $\norm{\z}^2
  = 2$, and $z_1z_2+ z_3z_4 + \dotsm + z_{2n-3}z_{2n-2} +
  z_{2n-1}z_{2n} = 0$, and the embedding $F$, can be written as
  \begin{equation*}
    F:\, (z_1,\dotsc,z_n) \mapsto \sqrt{2}\,(z_1, 0,\dotsc, z_n,0)\;.
  \end{equation*}

  If we restrict to the standard $3$--sphere, and its unit cotangent
  bundle, we can define the following embedding of $\S^3\times \C$
  into $\S(T^*\S^3)\subset \C^4$:
  \begin{equation*}
    \Phi:\, \bigl((z_1,z_2); w\bigr) \mapsto
    \frac{\sqrt{2}}{\sqrt{1+\abs{w}^2}}\,
    \bigl(z_1, -wz_2,z_2,wz_1\bigr)\;.
  \end{equation*}
  The pull-back of $\lcan = -\frac{i}{4}\, (\z\,d\bar\z - \bar \z\,
  d\z)$ by $\Phi$ gives
  \begin{equation*}
    \Phi^*\lcan = -\frac{i}{2}\,\bigl(z_1\,d\bar z_1 - \bar z_1\,d z_1
    + z_2\,d\bar z_2 - \bar z_2\,d z_2\bigr) -\frac{i}{2\,(1+\abs{w}^2)}\,
    \bigl(w\,d\bar w - \bar w\,d w\bigr)\; ,
  \end{equation*}
  and by combining this with the stretch map $\Disk \to \C, \,
  w\mapsto w/\sqrt{1-\abs{w}^2}$ we obtain the diffeomorphism that
  gives the desired pull-back
  \begin{equation*}
    \S^3\times \Disk \hookrightarrow \S(T^*\S^3), \,
    \bigl((z_1,z_2); w\bigr) \mapsto \Phi \Bigl((z_1,z_2);
    w/\sqrt{1-\abs{w}^2}\Bigr) =  \sqrt{2}\,
    \Bigl(\sqrt{1-\abs{w}^2}\,z_1,
    -wz_2,\sqrt{1-\abs{w}^2}\,z_2,wz_1\Bigr) \;.
  \end{equation*}
\end{example}

\begin{example}
  Let $(N,\alpha_N)$ be a closed contact manifold.  In
  \cite{BourgeoisTori} it was shown that using a compatible open book
  decomposition of $(N,\alpha_N)$, one finds functions $f,g:\, N\to\R$
  such that the $1$--form
  \begin{equation*}
    \alpha_\epsilon = \alpha_N + \epsilon\,\bigl(f\,dx + g\,dy\bigr)\;,
  \end{equation*}
  with $\epsilon\ne 0$ defines a contact structure on $N\times\T^2$,
  and for different values of $\epsilon$ all of these structures are
  equivalent.  Obviously the fibres $N\times\{(x,y)\}$ are
  contactomorphic to $(N,\alpha_N)$ so that we have a contact
  fibration $N\times \T^2 \to \T^2$.

  We will now show that the radius of the neighborhood of a fibre is
  $\infty$ that means we can embed $\bigl(N\times \R^2, \alpha_N +
  r^2\,d\phi\bigr)$ into $N\times \T^2$ such that $N\times\{0\}$
  corresponds to the fibre $N\times\{(x,y)\}$.  The proof works in
  several steps.  First we embed a disk $\Disk[\delta]$ into the base
  manifold $\T^2$.  We take the radial rays and lift them using the
  parallel transport.  This already will provide use a model $N\times
  \Disk[\delta]$ on which the contact form will be $\alpha_\epsilon =
  F_\epsilon\, \alpha_N + G_\epsilon\,d\phi$ with certain functions
  $F_\epsilon, G_\epsilon:\, N\times \Disk[\delta] \to \R$.  The
  reason why $\alpha_\epsilon$ has no $dr$--component after the
  transformation is that the rays on the disk follow the parallel
  transport.  The function $F_\epsilon$ does not vanish, and we can
  divide by it to obtain the new contact form $\alpha_N + \widehat
  G_\epsilon\,d\phi$ that already comes close to the desired form
  $\alpha_N + r^2\,d\phi$.  In fact, one obtains the desired form from
  our one, if one pulls back with the inverse of the map
  \begin{equation*}
    \bigl(p; re^{i\phi}\bigr)\in N\times \Disk[\epsilon] \mapsto
    \bigl(p;\, \bigl(\widehat G_\epsilon\bigr)^{1/2}\,e^{i\phi}\bigr) \;.
  \end{equation*}
  The reason that this is a diffeomorphism is that
  $\frac{\partial}{\partial r} \bigl(\widehat G_\epsilon\bigr)^{1/2}
  \ne 0$ everywhere, because otherwise the initial $1$--form could not
  be a contact form.  Now it suffices to look at the minimum of
  $\bigl(\widehat G_\epsilon\bigr)^{1/2}$ on $\partial\Disk[\delta]$.

  To perform the first step pull back $\alpha_\epsilon$ to $N\times
  \Disk[\delta]\times \Disk[\delta]$ using the map $N\times
  \Disk[\delta] \times \Disk[\delta] \to N\times \T^2, \,
  \bigl(p;(u,v);(x,y)\bigr) \mapsto (p;(x,y))$.  We will define a
  vector field $Y$ on $N\times \Disk[\delta] \times \Disk[\delta]$
  whose time~$1$ flow restricted to $N\times\Disk[\delta]\times \{0\}$
  will give the desired lift by the parallel transport.  Let
  $Y_\epsilon$ be
  \begin{equation*}
    Y_\epsilon\bigl(p;(u,v);(x,y)\bigr) = u\,\partial_x + v\,\partial_y
    - \epsilon\,(uX_f + vX_g)\,
  \end{equation*}
  where $X_f$ and $X_g$ are the contact vector fields for the
  functions $f$ and $g$ on $(N,\alpha_N)$ respectively, i.e., $X_f$
  and $X_g$ are vector fields parallel to the fibers that satistfy the
  equations
  \begin{equation*}
    \alpha_N(X_f) = f \quad\text{ and }\quad
    d\alpha_N(X_f,-) = -df + \bigl(\lie{\Reeb}f\bigr)\,\alpha_N \;,
  \end{equation*}
  and analogous equations for $X_g$ and the function $g$.  In
  particular it follows that $i_{Y_\epsilon}\alpha = 0$.  We want to
  pull back $\alpha_\epsilon$ with the time $1$ flow of $Y_\epsilon$.
  \begin{equation*}
    \begin{split}
      \bigl(\Phi^1_{Y_\epsilon}\bigr)^* \alpha_\epsilon -
      \alpha_\epsilon &= \int_0^1
      \left(\frac{d}{dt}\bigl(\Phi^t_{Y_\epsilon}\bigr)^*
        \alpha_\epsilon\right)\,dt = \int_0^1 \left(
        \bigl(\Phi^t_{Y_\epsilon}\bigr)^*
        \lie{Y_\epsilon}\alpha_\epsilon\right)\, dt = \int_0^1 \left(
        \bigl(\Phi^t_{Y_\epsilon}\bigr)^*
        \iota_{Y_\epsilon} d \alpha_\epsilon\right)\, dt \\
      & = \int_0^1 \bigl(\Phi^t_{Y_\epsilon}\bigr)^* \left(
        \iota_{Y_\epsilon} d \alpha_N +
        \epsilon\,\iota_{Y_\epsilon} d(f\,dx + g\,dy)\right)\, dt\\
      &= \epsilon\, \int_0^1 \bigl(\Phi^t_{Y_\epsilon}\bigr)^*\left( -
        d \alpha_N(uX_f + vX_g,-) + \iota_{Y_\epsilon}(df\wedge dx +
        dg\wedge dy)\right)\, dt \\
      &= \epsilon\, \int_0^1 \bigl(\Phi^t_{Y_\epsilon}\bigr)^*\left(
        u\,df + v\,dg - \lie{\Reeb} (uf + vg) \, \alpha_N +
        \iota_{Y_\epsilon}(df\wedge dx + dg\wedge dy)\right)\, dt \\
      &= \epsilon\, \int_0^1 \bigl(\Phi^t_{Y_\epsilon}\bigr)^*\left(
        u\,df + v\,dg - \lie{\Reeb} (uf + vg) \, \alpha_N - u\,df -
        v\,dg -\epsilon\,\bigl((uX_f + vX_g) f\bigr)\, dx
        -\epsilon\,\bigl((uX_f + vX_g) g\bigr)\, dy \right)\, dt \\
      &= \epsilon\, \int_0^1 \bigl(\Phi^t_{Y_\epsilon}\bigr)^*\left( -
        \lie{\Reeb} (uf + vg) \, \alpha_N -\epsilon\,\bigl((uX_f +
        vX_g) f\bigr)\, dx -\epsilon\,\bigl((uX_f + vX_g) g\bigr)\, dy \right)\, dt \\
      &= \epsilon\, \int_0^1 \bigl(\Phi^t_{Y_\epsilon}\bigr)^*\Bigl( -
      \lie{\Reeb} (uf + vg) \, \alpha_N -\epsilon\,\bigl(u
      f\lie{\Reeb}f +
      vf\lie{\Reeb}g - v\,\alpha_N([X_f,X_g])\bigr)\, dx \\
      & -\epsilon\,\bigl(vg\lie{\Reeb} g + u g\lie{\Reeb} f + u\,\alpha_N ([X_f,X_g])\bigr)\, dy \Bigr)\, dt \\
    \end{split}
  \end{equation*}

\end{example}
}

\section{Preliminaries}

\subsection{Fillability}\label{sec: definition fillability}

In this section, we will briefly present some standard definitions and
properties regarding fillability and $J$--holomorphic curves.

\begin{defi}
  A \textbf{Liouville field} $X_L$ is a vector field on a symplectic
  manifold $(W,\omega)$ for which
  \begin{equation*}
    \lie{X_L} \omega = \omega
  \end{equation*}
  holds.
\end{defi}

If $(W,\omega)$ is a symplectic manifold with boundary $M := \partial
W$, and if $X_L$ is a Liouville field on $W$ that is transverse to
$M$, then the kernel of the $1$--form
\begin{equation*}
  \alpha := \left.\omega(X_L,-)\right|_{TM}
\end{equation*}
defines a contact structure on~$M$.

\begin{defi}
  Let $(M,\xi)$ be a closed contact manifold.  A compact symplectic
  manifold $(W,\omega)$ with boundary $\partial W =M$ is called a
  \textbf{strong (symplectic) filling} of $(M,\xi)$, if there exists a
  Liouville field $X_L$ in a neighborhood of the boundary $M$ pointing
  \emph{outwards} through $M$ such that $X_L$ defines a contact form
  for $\xi$.  If the vector field $X_L$ is defined globally on $W$, we
  speak of an \textbf{exact symplectic filling}.
\end{defi}

\begin{remark}
  In a symplectic filling, we can always find a neighborhood of $M$
  that is of the form $(-\epsilon,0]\times M$ by using the negative
  flow of $X_L$ to define
  \begin{equation*}
    (-\epsilon,0]\times M \to W,\, (p,t) \mapsto \Phi_t(p)\;.
  \end{equation*}
  Denote the hypersurfaces $\{t\}\times M$ by $M_t$, and the $1$--form
  $\omega(X_L,-)$ by $\widehat\alpha$.  It is clear that
  $\widehat\alpha$ defines on every hypersurface $M_t$ a contact
  structure.  The Reeb field $\Reeb$ is the unique vector field on
  $(-\epsilon, 0]\times M$ that is tangent to the hypersurfaces $M_t$,
  and satisfies both $\omega(\Reeb,Y) = 0$ for every $Y\in TM_t$, and
  $\omega(X_L,\Reeb) = 1$.  This field restricts on any hypersurface
  $M_t$ to the usual Reeb field for the contact form
  $\left.\widehat\alpha\right|_{TM_t}$.

  Below we will show that the ``height'' function $h:\,
  (-\epsilon,0]\times M\to \R, (t,p) \mapsto t$ is plurisubharmonic
  with respect to certain almost complex structures.
\end{remark}

In the context of this article we will use the term ``adapted almost
complex structure'' in the following sense.

\begin{defi}
  Let $(W,\omega)$ be a symplectic filling of a contact manifold
  $(M,\alpha)$.  An almost complex structure $J$ is a smooth section
  of the endomorphism bundle $\End(TW)$ such that $J^2 = -\1$.  We say
  that $J$ is \textbf{adapted to the filling}, if it is compatible
  with $\omega$ in the usual sense, which means that for all $X,Y\in
  T_pW$
 \begin{align*}
   \omega(JX,JY) &= \omega(X,Y)
   \intertext{holds, and}
   g(X,Y) &:= \omega(X,JY)
 \end{align*}
 defines a Riemannian metric.  Additionally, we require $J$ to satisfy
 close to the boundary $M = \partial W$ the following properties: For
 the two vector fields $X_L$ and $\Reeb$ introduced above, $J$ is
 defined as
 \begin{align*}
   JX_L = \Reeb \text{ and } J\Reeb = - X_L\;,
 \end{align*}
 and $J$ leaves the subbundle $\xi_t = TM_t \cap \ker \widehat\alpha$
 invariant.
\end{defi}

\begin{proposition}
  Let $V$ be an open subset of $\C$, and let $u:\,V \to W$ be a
  $J$--holomorphic map.  The function $h\circ u:\,V \to \R$ is
  subharmonic.
\end{proposition}
\begin{proof}
  A short computation shows that $\widehat\alpha = -dh \circ J$, and
  then we get
  \begin{equation*}
    \begin{split}
      0 & \le u^*\omega = u^*d\iota_{X_L}\omega = u^*d\widehat \alpha =
      u^* d \bigl(- dh\circ J\bigr) = -u^* dd^c h \\
      & = -dd^c(h\circ u) = \Bigl(\frac{\partial^2 h\circ u}{\partial
        x^2}+ \frac{\partial^2 h\circ u}{\partial y^2}
      \Bigr)\,dx\wedge dy\;.  \qedhere
    \end{split}
  \end{equation*}
\end{proof}

\begin{corollary}\label{kurven transvers zu rand}
  By the strong maximum principle and the boundary point lemma (e.g.\
  \cite{GilbargTrudinger}), any $J$--holomorphic curve
  $u:\,(\Sigma,\partial\Sigma) \to (W,\partial W)$ is either constant
  or it touches $M = \partial W$ only at its boundary
  $\partial\Sigma$, and this intersection is transverse.  Furthermore,
  if $u$ is non constant, then the boundary curve
  $\left.u\right|_{\partial\Sigma}$ has to intersect the contact
  structure $\xi$ on $\partial W$ in positive Reeb direction.
\end{corollary}

In the rest of the article, we will denote the half space
$\bigl\{z\in\C\bigm|\,\ImaginaryPart z\ge 0\bigr\}$ by
$\HyperbolicPlane$.  Let $\phi:\, N\looparrowright M$ be an immersion
of a manifold $N$ in $M$.  We define the self-intersection set of $N$
as
\begin{equation*}
  \selfintersection{N} := \bigl\{p\in N\bigm|\,
  \text{$\exists p'\ne p$ with $\phi(p) = \phi(p')$}\bigr\} \;.
\end{equation*}

\subsection{Tubular neighborhood theorem for contact submanifolds}

Let $N$ be a contact submanifold of $(M, \alpha)$.  The contact
structure $\xi = \ker \alpha$ can be split along $N$ into the two
subbundles
\begin{equation*}
  \left.\xi\right|_N = \xi_N \oplus \xi_N^\perp\;,
\end{equation*}
where $\xi_N$ is the contact structure on $N$ given by $\xi_N = TN
\cap \left.\xi\right|_N = \ker \left.\alpha\right|_{TN}$, and
$\xi_N^\perp$ is the symplectic orthogonal of $\xi_N$ inside
$\left.\xi\right|_N$ with respect to the form $d\alpha$.  Note that
$\xi_N^\perp$ carries a conformal symplectic structure given by
$d\alpha$, but neither $\xi_N^\perp$ nor the conformal symplectic
structure do depend on the contact form chosen on $M$.  The bundle
$\xi_N^\perp$ can be identified with the normal bundle of $N$.

A well known neighborhood theorem states that $\xi_N^\perp$ determines
a small neighborhood of $N$ completely.

\begin{theorem}
  Let $(N,\xi_N)$ be a contact submanifold of both $(M_1, \xi_1)$ and
  $(M_2, \xi_2)$.  Assume that the two normal bundles
  $(\xi_1)_N^\perp$ and $(\xi_2)_N^\perp$ are isomorphic as conformal
  symplectic vector bundles.  Then there exists a small neighborhood
  of $N$ in $M_1$ that is contactomorphic to a small neighborhood of
  $N$ in $M_2$.
\end{theorem}

If $N$ has a trivial conformal symplectic normal bundle $\xi_N^\perp$,
then we call the product space $N\times \R^{2k}$ with contact
structure $\alpha_N + \sum_{j=1}^k (x_j\,dy_j - y_j\,dx_j)$ the
\emph{standard model for neighborhoods} of $N$.

\section{The generalized plastikstufe (GPS)}\label{sec: definition GPS}

\begin{defi}
  Let $(M,\alpha)$ be a $(2n+1)$--dimensional contact manifold, and
  let $S$ be a closed $(n-1)$--dimensional manifold.  A
  \textbf{generalized plastikstufe} (\textbf{GPS}) is an immersion
  \begin{equation*}
    \Phi:\, S\times \Disk \looparrowright M, \,
    \bigl(s,re^{i\phi}\bigr) \to \Phi\bigl(s,re^{i\phi}\bigr) \;,
  \end{equation*}
  such that the pull-back $\Phi^*\alpha$ reduces to the form
  $f(r)\,d\phi$ with $f\ge 0$ that only vanishes for $r=0$, and $r=1$.
  Furthermore there is an $\epsilon>0$ such that self-intersections
  may only happen between points of the form $(s,re^{i\phi})$, and
  $(s',r'e^{i\phi})$ with $r,r'\in (\epsilon, 1-\epsilon)$ that have
  equal $\phi$--coordinate.  Finally there must be an open set joining
  $S\times\{0\}$ with $S\times\partial\Disk$ that does not contain any
  self-intersection points.

  We call $S\times\{0\}$ (or also its image) the \textbf{core} of the
  GPS, and $S\times \partial\Disk$ (or again the image) its
  \textbf{boundary}.  We denote $S\times \bigl(\Disk - \{0\} -
  \partial\Disk\bigr)$ by $\GPS^*$ and call it the \textbf{interior}
  of the GPS.
\end{defi}

\begin{remark}\label{remark: label for leaves}
  The regular leaves of the GPS are given by the sets $\{\phi =
  \mathrm{const}\}$.  We are hence requiring that self-intersections
  only happen between points lying on the same leaf.  A different way
  to state this requirement consists in saying that there is a
  continuous map
  \begin{equation*}
    \theta:\, \Phi\bigl(\GPS^*\bigr)\to \S^1
  \end{equation*}
  such that $\theta\bigl(\Phi(s,r,\phi)\bigr) = \phi$.
\end{remark}

\begin{theorem}\label{theorem: main result}
  A closed contact manifold $(M,\alpha)$ that contains a GPS does not
  have an exact symplectic filling.
\end{theorem}

\begin{remark}
  Using a more precise analysis of bubbling (as in
  \cite{BubblingImmersedBoundary}) should make it possible to prove
  that a $\GPS$ is an obstruction to finding even a (semipositive)
  strong symplectic filling.  In Remark~\ref{remark: sketch full
    proof}, we sketch how the proof would have to be modified.  Note
  though that \cite{BubblingImmersedBoundary} requires that the
  self-intersections of the GPS are clean.
\end{remark}

\section{Constructing immersed plastikstufes in neighborhoods of
  submanifolds}\label{sec: construction of GPS in tubular nbhd}

\subsection{Local construction in codimension two}

The most prominent example of an overtwisted contact manifold in the
literature is $\R^3$ with the structure induced by the contact form
\begin{equation*}
  \alpha_- = \cos r\, dz + r\sin r\,d\phi \;,
\end{equation*}
written in cylindrical coordinates $(r,\phi,z)$ such that $x=
r\cos\phi$, $y=r\sin\phi$, and $z = z$.  Any plane $\{z =
\text{const.}\}$ contains an overtwisted disk centered at the origin
with radius $r = \pi$.  From the classification in
\cite{EliashbergContactR3}, it follows that $(\R^3,\alpha_-)$ is up to
contactomorphism the unique contact structure on $\R^3$ that is
overtwisted at infinity, and hence any sufficiently small contractible
neighborhood of an overtwisted disk in a contact $3$--manifold is
contactomorphic to $(\R^3,\alpha_-)$.

The Reeb field $\Reeb$ associated to $\alpha_-$ is given by
\begin{equation*}
  \Reeb = \frac{1}{r+\sin r\,\cos r}\,\bigl(\sin r\,\partial_\phi
  + (\sin r + r\cos r)\,\partial_z\bigr)\;.
\end{equation*}
Its flow $\Phi_t$ is linear, because $r$ remains constant, and the
coefficients of the $z$-- and the $\phi$--coordinate only depend on
the $r$--coordinate.  The Reeb field is tangent to the overtwisted
disk on the circle of radius $r_0$ such that $r_0 = - \tan r_0$
($r_0\approx 2.029$).  Inside this circle $\Reeb$ has a positive
$z$--component, outside it has a negative one.  This means the
overtwisted disk $\overtwisted$ and its translation by the Reeb flow
$\Phi_t(\overtwisted)$ for a time $t\ne 0$ only intersect along the
circle of radius $r_0$ (see Fig.~\ref{fig: ot disk and its image by
  flow}).  More precisely, the Reeb field reduces on the circle of
radius $r_0$ to $\Reeb = 1/r_0\sin r_0 \,\partial_\phi$, so that it
defines a rotation with period $T = 2\pi r_0 \sin r_0 \approx 11.4$.

\begin{figure}[htbp]
 \begin{picture}(0,0)%
\includegraphics{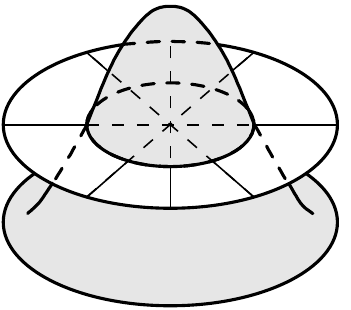}%
\end{picture}%
\setlength{\unitlength}{3947sp}%
\begingroup\makeatletter\ifx\SetFigFontNFSS\undefined%
\gdef\SetFigFontNFSS#1#2#3#4#5{%
  \reset@font\fontsize{#1}{#2pt}%
  \fontfamily{#3}\fontseries{#4}\fontshape{#5}%
  \selectfont}%
\fi\endgroup%
\begin{picture}(1673,1474)(3420,-4287)
\put(4953,-3180){\makebox(0,0)[lb]{\smash{{\SetFigFontNFSS{10}{12.0}{\familydefault}{\mddefault}{\updefault}{\color[rgb]{0,0,0}$\overtwisted$}%
}}}}
\put(5078,-3961){\makebox(0,0)[lb]{\smash{{\SetFigFontNFSS{10}{12.0}{\familydefault}{\mddefault}{\updefault}{\color[rgb]{0,0,0}$\Phi_t\bigl(\overtwisted\bigr)$}%
}}}}
\end{picture}%
  \caption{The overtwisted disk and its image under the Reeb flow only
    intersect along a circle of radius~$r_0$.}\label{fig: ot disk and
    its image by flow}
\end{figure}
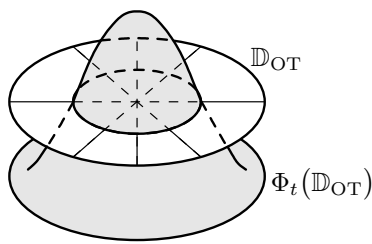

Take now the product of $\R^3$ with $\R^2$, and define on $\R^3\times
\R^2$ the contact form $\alpha_- + R^2\,d\theta$, where $(R,\theta)$
are polar coordinates of $\R^2$.  This is a contact fibration, and we
will use the first step of the construction in
\cite{PresasExamplesPlastikstufes}, namely we will trace a closed path
$\gamma:\,\S^1 \to \R^2$ that has the shape of a figure-eight, with
the double point at the origin, and such that both parts of the eight
have equal area with respect to the standard area form $2R\,dR\wedge
d\theta$.  Start at the origin of the disk, at $\gamma(1) = 0$ on this
closed loop, and regard the overtwisted disk $\overtwisted$ in the
fiber $\R^3\times\{0\}$ described above.  By using the parallel
transport of $\overtwisted$ along the path $\gamma$, we are able to
describe an immersed plastikstufe.  The parallel transport reduces in
the fibers to the flow of the vector field $-c\,\Reeb$ with $c =
\norm{\gamma}^2\, d\theta(\gamma')$, so that the monodromy of a closed
loop is just given by the Reeb flow $\Phi_T$ for a time $T$ that is
equal to the area that has been enclosed by the loop, where we have to
count with orientation.  The total area of the figure-eight $\gamma$
vanishes, because on one part of the eight, we are turning in positive
direction, on the other in the opposite one, and the area of both
parts was chosen to be equal.

\begin{figure}[htbp]
\begin{picture}(0,0)%
\includegraphics{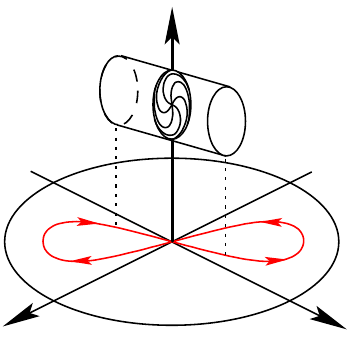}%
\end{picture}%
\setlength{\unitlength}{3947sp}%
\begingroup\makeatletter\ifx\SetFigFontNFSS\undefined%
\gdef\SetFigFontNFSS#1#2#3#4#5{%
  \reset@font\fontsize{#1}{#2pt}%
  \fontfamily{#3}\fontseries{#4}\fontshape{#5}%
  \selectfont}%
\fi\endgroup%
\begin{picture}(1677,1732)(3413,-3976)
\put(5045,-3759){\makebox(0,0)[lb]{\smash{{\SetFigFontNFSS{9}{10.8}{\familydefault}{\mddefault}{\updefault}{\color[rgb]{0,0,0}$x$}%
}}}}
\put(3453,-3916){\makebox(0,0)[lb]{\smash{{\SetFigFontNFSS{9}{10.8}{\familydefault}{\mddefault}{\updefault}{\color[rgb]{0,0,0}$y$}%
}}}}
\put(4285,-2367){\makebox(0,0)[lb]{\smash{{\SetFigFontNFSS{10}{12.0}{\familydefault}{\mddefault}{\updefault}{\color[rgb]{0,0,0}$\R^3$}%
}}}}
\end{picture}%
  \caption{Parallel transport of the overtwisted disk along a
    figure-eight path yields an immersed plastikstufe.}\label{fig:
    transport along figure eight}
\end{figure}
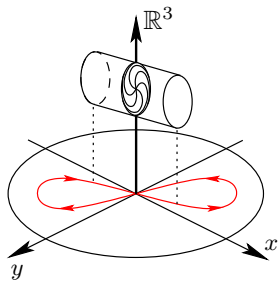

We will describe the construction more explicitely to better
understand the self-intersection set.  The parallel transport of the
overtwisted disk defines an immersion
\begin{equation*}
  \overtwisted \times \S^1 \to \R^3\times \R^2,\,
  \bigl((x,y,0),e^{i\theta}\bigr) \mapsto \bigl(\Phi_{T(\theta)}
  (x,y,0), \gamma(e^{i\theta})\bigr)\;,
\end{equation*}
where $T(\theta) = \int_\gamma r^2\,d\phi$.  The map is well defined,
because $T(\theta + 2\pi) = T(\theta)$.  It is also easy to see that
this map is an immersion.

The only self-intersection points may lie over the crossing $\gamma(1)
= \gamma(-1)$ in the figure-eight, and in fact, since the Reeb flow
moves the interior of the overtwisted disk up, and the outer part
down, self-intersections only happen between the two circles
\begin{equation*}
  \bigl\{(r_0\,\cos \phi, r_0\,\sin\phi, 0)\bigr\} \times \{-1, 1\}
  \subset \overtwisted \times \S^1 \;.
\end{equation*}
Denote the area enclosed by one of the petals of the figure-eight path
by $A$.  The images of any pair of points $\bigl((r_0\,\cos \phi,
r_0\,\sin\phi, 0), 1\bigr)$ and $\bigl((r_0\,\cos (\phi-t_0),
r_0\,\sin(\phi-t_0), 0), -1\bigr)$, with $t_0 = A/r_0\sin r_0$ are
identical.

Note that if $\gamma$ is chosen such that $A = 2\pi r_0\sin r_0$, then
the pair of points that intersect each other always lie on the same
ray of the overtwisted disk, and we have in fact constructed a GPS.
The figure-eight path has to enclose a sufficiently large area, and we
cannot realize such a path $\gamma$ in a disk $\Disk[R]$ of radius $R
< 2\sqrt{r_0\sin r_0} \approx 3.82$.

\subsection{Higher codimension}

Use now the same contact structure on $(\R^3,\alpha_-)$ as above, and
extend it to a contact structure on $\R^3\times \C^k$ with contact
form
\begin{equation*}
  \alpha_- + \sum_{j=1}^k r_j^2\,d\phi_j\;,
\end{equation*}
where $(r_j,\phi_j)$ are polar coordinates for the $j$--th
$\C$--factor in $\C^k$.

Now we will take the $k$--fold product of figure-eight loops of
different sizes, and group them into a map
\begin{equation*}
  \Gamma:\,\T^k \looparrowright \C^k,\, (e^{i\theta_1},\dotsc,e^{i\theta_k})
  \mapsto \bigl(\gamma(e^{i\theta_1}), 2^{1/2}\,\gamma(e^{i\theta_2}),
  \dotsc, 2^{(k-1)/2}\,\gamma(e^{i\theta_k})\bigr)\;.
\end{equation*}
This map is an immersion with self-intersection set 
\begin{equation*}
  \selfintersection{\Gamma} = \bigl\{(e^{i\theta_1},\dotsc,e^{i\theta_k})
  \in \T^k\bigm|\,\text{at least one of the  $\theta_j$ lies in
    $\pi\Z$}\bigr\}\;.
\end{equation*}
Define functions $T_j(\theta):= 2^{j-1}\,\int_0^\theta
\gamma^*\bigl(r^2\,d\phi\bigr)$, and
$T(e^{i\theta_1},\dotsc,e^{i\theta_k}) = \sum_{j=1}^k T_j(\theta_j)$.
Then the immersion
\begin{equation*}
  \overtwisted \times \T^k \to \R^3\times \C^k,\,
  \bigl((x,y,0);e^{i\theta_1},\dotsc,e^{i\theta_k}\bigr)
  \mapsto \bigl(\Phi_{T(\theta_1,\dotsc,\theta_k)}
  (x,y,0); \Gamma(e^{i\theta_1},\dotsc,e^{i\theta_k})\bigr)\;,
\end{equation*}
where $\Phi_t$ denotes the Reeb flow, is a GPS.  Obviously the
self-intersection points of this map are contained in the preimage of
the self-intersection set $\selfintersection{\Gamma}$ downstairs.
Consider two points $(e^{i\theta_1},\dotsc,e^{i\theta_k})$ and
$(e^{i\psi_1},\dotsc,e^{i\psi_k})$ that have the same image under
$\Gamma$.  It follows for each pair $(\theta_j,\psi_j)$ that either
$\theta_j = \psi_j$ or that $\theta_j, \psi_j\in \pi\Z$.  The disks
lying over such points are given by $\Phi_{T(\mathbf{\theta})}
(\overtwisted)$ and $\Phi_{T(\mathbf{\psi})}(\overtwisted)$
respectively, where $\overtwisted = \bigl\{(x,y,0)\bigm|\,
\norm{(x,y,0)} \le \pi\bigr\}$.  The Reeb flow is $\phi$--invariant
and preserves the distance of the points $(x,y,0)$ from the $z$--axis.
Hence in the interior and the exterior of the circle of radius $r_0$,
$\Phi_{T(\mathbf{\theta})}(x,y,0)$ can only be equal to
$\Phi_{T(\mathbf{\psi})}(x',y',0)$, if $T(\mathbf{\theta}) =
T(\mathbf{\psi})$, because the flow changes the $z$--coordinate, and
by the coefficients chosen in $\Gamma$ for the paths, $T$ is injective
on $(\pi a_1,\dotsc, \pi a_k)$ with all $a_j\in \{0,1\}$.
Additionally then we have $(x,y,0) = (x',y',0)$, so that no
self-intersections can happen for points with $\norm{(x,y,0)} \ne
r_0$.  Self-intersections of the GPS can hence only exist for points
where the distance of $(x,y)$ from the origin is equal to $r_0$, but
by the size condition on the figure-eight loops the holonomy will
always correspond to a rotation by a multiple of $2\pi$ so that all
conditions of a GPS are satisfied by this map.

\subsection{Application to contact submanifolds}

Let $(N,\alpha_N)$ be an overtwisted contact $3$--manifold.  We will
show that the product manifold $\bigl(N\times\C^k, \alpha_N +
\sum_{j=1}^k r_j^2\,d\theta_j\bigr)$, where $(r_j,\theta_j)$ are polar
coordinates on the $j$--th factor of $\C^k$ contains a GPS.

Consider a small contractible neighborhood of an overtwisted disk
$\overtwisted$ in $N$.  This neighborhood is contactomorphic to
$(\R^3,\alpha_-)$, because it is overtwisted at infinity.  Choose a
large ball $B$ in $\R^3$ (so large that the Reeb flow for $\alpha_-$
restricted to the overtwisted disk exists for long enough times), then
there is a function $f:\, N\to\R$ such that the chosen ball $B$ can be
embedded by a strict contactomorphism (that means preserving the
contact form) into $(N,f\,\alpha_N)$.  The contact form $f\,\alpha_N +
\sum_{j=1}^k f\,r_j^2\,d\theta_j$ on the product manifold
$N\times\C^k$ can be transformed by the map $(p;z_1,\dotsc,z_k)
\mapsto (p;z_1/\sqrt{f},\dotsc,z_k/\sqrt{f})$ into
\begin{equation*}
  \bigl(N\times\C^k, f\,\alpha_N +
  \sum_{j=1}^k r_j^2\,d\theta_j\bigr)\;.
\end{equation*}
This contains a subset of the form $\bigl(B\times \C^k, \alpha_- +
\sum_{j=1}^k r_j^2\,d\theta_j\bigr)$ in which we can perform the
contruction explained above.

\begin{corollary}\label{GPS in manifolds with overtwisted blah}
  Let $(M,\alpha)$ be a closed contact $(2n+1)$--manifold that
  contains an overtwisted contact submanifold $N$ of dimension $3$
  that has trivial contact normal bundle.  There is a neighborhood of
  $N$ that is contactomorphic to a neighborhood $U$ of $N\times\{0\}$
  in the product space $\bigl(N\times\C^k, \alpha_N + \sum_{j=1}^k
  r_j^2\,d\theta_j\bigr)$.

  If the neighborhood $U$ contains a sufficiently large disk bundle of
  $N\times\{0\}$, then it follows that $M$ does not admit an exact
  symplectic filling.
\end{corollary}
\begin{proof}
  By the construction just described $\bigl(N\times\C^k, \alpha_N +
  \sum_{j=1}^k r_j^2\,d\theta_j\bigr)$ contains a GPS.  Since the GPS
  is compact, it is contained in some disk bundle around
  $N\times\{0\}$.  If the neighborhood of $N$ is contactomorphic to
  this disk bundle, then $(M,\alpha)$ contains a GPS, and hence cannot
  have an exact symplectic filling.
\end{proof}

\section{Proof of Theorem~\ref{theorem: main result}}
\label{sec: proof of nonfillability}

\subsection{Sketch of the proof}

The proof is based on \cite{NiederkruegerPlastikstufe} (which in turn
is ultimately based on \cite{Eliashberg_NonFillable,Gromov_Kurven}),
and it is very helpful to have a good understanding of this first
article.  Assume that $(M,\alpha)$ has an exact symplectic filling
$(W,\omega)$.  We choose an adapted almost complex structure $J$ on
$W$ that has in a neighborhood of the core $S\times\{0\}$ the special
form described in \cite[Section~3]{NiederkruegerPlastikstufe}, and in
a neighborhood of the boundary $S\times\partial\Disk$ the particular
form described in Section~\ref{sec: boundary of GPS} below.

The chosen complex structure allows us to write down explicitely the
members of a Bishop family around the core of the GPS, so that we find
a non-empty moduli space $\moduli$ of holomorphic disks $u:\, (\Disk,
\partial\Disk) \to (M, \GPS^*)$ with a marked point $z_0\in
\partial\Disk$.  The boundary of each holomorphic disk $u$ intersects every
regular leaf of the GPS exactly once, or expressed differently the
following composition defines a diffeomorphism $\left.\theta\circ
  u\right|_{\partial\Disk}:\, \S^1 \to \S^1$ on the circle.  The
Bishop family is canonically diffeomorphic to a neighborhood of the
core $S\times \{0\}$ via the evaluation map
\begin{equation*}
  \evaluationmap_{z_0}: \, \moduli \to \GPS^*,\, u\mapsto u(z_0)\;.
\end{equation*}

We can now apply similar intersection arguments for the boundary
$S\times\partial\Disk$ of the GPS (Section~\ref{sec: boundary of
  GPS}), and for the core
(\cite[Section~3]{NiederkruegerPlastikstufe}), showing that there
exists a neighborhood of $\partial\GPS$ that cannot be penetrated by
any holomorphic disk, and that the only disks that come close to the
core are the ones lying in the Bishop family.

Choose now a smooth generic path $\gamma \subset S\times\Disk$ that
avoids the self-intersection points of the GPS, and that connects the
core $S\times \{0\}$ with the boundary of the GPS.  In
Section~\ref{sec: moduli space}, we define the moduli space
$\moduli_\gamma := \evaluationmap_{z_0}^{-1}(\gamma)$, and show that
it is a smooth $1$--dimensional manifold.  From now on, we will
further restrict $\moduli_\gamma$ to the connected component of the
moduli space that contains the Bishop family.  Then in fact
$\moduli_\gamma$ has to be diffeomorphic to an open interval.  The
compactification of one of the ends of the interval simply consists in
decreasing the size of the disks in the Bishop family until they
collapse to a single point at $\gamma(0)$ on the core of the GPS.

Our aim will be to understand the possible limits at the other end of
the interval $\moduli_\gamma$, and to deduce a contradiction to the
fillability of $M$.  The energy of all disks $u\in\moduli_\gamma$ is
bounded by $2\pi\,\max f$, where $\alpha = f(r)\,d\phi$ on the GPS.
By requiring that the GPS has only clean intersections, we could apply
the compactness theorem in \cite{BubblingImmersedBoundary} to deduce
even a contradiction for the existence of a semipositive filling (see
Remark~\ref{remark: sketch full proof}).  Instead of merely referring
to that result, we have decided to give a full proof of compactness in
our situation (see Section~\ref{sec: bubbling}).  This way we can drop
the stringent conditions on the self-intersections of the GPS, and the
required proof is in fact significantly simpler than the full proof of
the compactness theorem.

It then follows that for any sequence of disks $u_k\in\moduli_\gamma$,
we find a family of reparametrizations $\phi_k:\, \Disk \to \Disk$
such that $u_k\circ \phi_k$ contains a subsequence converging
uniformly with all derivatives to a disk $u_\infty \in\moduli_\gamma$.
This means that $\moduli_\gamma$ is compact, but since at the same
time we know that the far-most right element $u_\infty$ in
$\moduli_\gamma$ has a small neighborhood in $\moduli_\gamma$
homeomorphic to an open interval, it follows that $u_\infty$ is not a
boundary point of $\moduli_\gamma$.  Compactness contradicts thus the
existence of the filling.

\begin{remark}\label{remark: sketch full proof}
  We will briefly sketch how \cite{BubblingImmersedBoundary} could be
  used to prove the non-existence of even a semipositive filling, if
  the GPS is cleanly immersed.

  The limit of a sequence of holomorphic disks can be described as the
  union of finitely many holomorphic spheres $u_S^1, \dotsc, u_S^K$
  and finitely many holomorphic disks $v_1,\dotsc, v_N$.  The
  holomorphic disks $v_j:\, (\Disk,\partial\Disk) \to (W,\GPS)$ are
  everywhere smooth with the possible exception of boundary points
  that lie on self-intersections of the GPS.  Here $v_j$ will still be
  continuous though (As an example of such disks, take a figure-eight
  path in the complex plane $\C$.  By the Riemann mapping theorem,
  there is a holomorphic disk enclosed into each of the loops, but
  obviously these disks cannot be smooth on their boundary at the
  self-intersection point of the eight).

  We will now first prove that the limit curve of a sequence in
  $\moduli_\gamma$ is only composed of a single disk, which then
  necessarily has to be smooth.  Assume we would have a decomposition
  into several disks $v_1,\dotsc, v_N$.  The boundary of each of these
  disks $v_j$ is a continuous path in $\GPS^*$, we can hence combine
  the disks with the projection $\theta$ defined in
  Remark~\ref{remark: label for leaves} to obtain a continuous map
  $\left.\theta\circ v_j\right|_{\partial\Disk}:\, \S^1 \to \S^1$.
  Thus we can associate to each of the disks $v_j$ a degree.  In fact
  it follows that $\deg \left.\theta\circ v_j\right|_{\partial\Disk} >
  0$, because almost all points on the boundary of $v_j$ are smooth,
  and for them $v_j$ has to intersect, by Corollary~\ref{kurven
    transvers zu rand}, all leaves of the foliation of the GPS in
  positive direction.  Finally assume that there are still several
  disks, each one necessarily with $\deg \left.\theta\circ
    v_j\right|_{\partial\Disk} \ge 1$.  This means that the
  composition of the maps $\left.\theta\circ v_j\right|_{\partial
    \Disk}$ will cover the circle several times, but this is not
  possible for the limit of injective maps $\left.\theta\circ
    u_k\right|_{\partial\Disk}$.  There is hence only a single disk in
  the limit.  Using Theorem~\ref{theorem: convergence for bounded
    maps} below it finally also follows that this disk is smooth, and
  has a boundary that lifts to a smooth loop in $S\times\Disk$.
  
  The reason why there are no holomorphic spheres as bubbles is a
  genericity argument, since the disk and all spheres are regular
  smooth holomorphic objects, we can compute the dimension of the
  bubble tree in which our limit object would lie.  By the assumption
  of semi-positivity, it follows that the dimension would be negative.
\end{remark}

\subsection{The moduli space} \label{sec: moduli space}

The aim of this section is to define the moduli space of holomorphic
disks and to prove that it is a smooth manifold.  Care has to be
taken, because the boundary condition considered in this article is
not a properly embedded, but only an immersed submanifold.  The main
idea is to restrict to those holomorphic curves whose boundary lies
locally always on a single leaf of the immersed submanifold.  We can
then easily adapt standard results.

Let $(W, J)$ be an almost complex manifold, and let $L$ be a compact
manifold with $2 \dim L = \dim W$.

\begin{defi}
  An \textbf{immersed totally real submanifold} is an immersion
  $\phi:\, L \looparrowright W$ such that
  \begin{equation*}
    \bigl(D\phi\cdot T_xL\bigr)
    \oplus \bigl(J\cdot D\phi\cdot T_xL\bigr) = T_{\phi(x)} W
  \end{equation*}
  at every $x\in L$.
\end{defi}

Let $\phi:\, L \looparrowright W$ be a totally real immersed
submanifold with self-intersection set $\selfintersection{L}$.  Choose
a (not necessarily connected) submanifold $A\hookrightarrow L$ that is
disjoint from $\selfintersection{L}$.  Let $\Sigma$ be a Riemann
surface with $N$ boundary components $\partial\Sigma_1, \dotsc,
\partial\Sigma_N$, and choose on each boundary component a marked
point $z_j \in \partial\Sigma_j$.  Then define $\mathcal{B}(\Sigma; L;
A)$ to be the set of maps
\begin{equation*}
  u:\, (\Sigma, \partial\Sigma) \mapsto \bigl(W, \phi(L)\bigr)
\end{equation*}
for which the boundary circles $\left.u\right|_{\partial\Sigma}$ can
be lifted to continuous loops $c:\, \partial \Sigma \to L$ such that
$\phi\circ c = \left.u\right|_{\partial\Sigma}$, and $c(z_j) \in A$.

Note that with our conditions the lift of the boundary circles
$\left.u\right|_{\partial\Sigma}$ is unique, because if there were two
different loops $c,c':\, \partial \Sigma_j \to L$ with $\phi\circ c =
\phi\circ c'$, and $c(z_j),c'(z_j) \in A$, it follows that the set
$\bigl\{z\in\partial\Sigma_j\bigm|\, c(z) = c'(z)\bigr\}$ contains the
point $z_j$, and is hence non-empty.  Furthermore this set is closed,
because it is the preimage of the diagonal $\triangle_W := \{(p,p)|\,
p\in W\}$ under the map $c\times c': \partial \Sigma_j\times \partial
\Sigma_j \to W \times W$ intersected with the diagonal
$\triangle_{\partial \Sigma_j}$.  Finally, $L$ can be covered by open
sets on each of which the immersion $\phi$ is injective, and hence if
$c(z) = c'(z)$ there is also an open neighborhood of $z$ on which both
paths coincide.  It follows that $c$ and $c'$ are equal.

We have to prove that $\mathcal{B}(\Sigma; L; A)$ is a Banach manifold
by finding a suitable atlas.  To define a chart around a map $u_0\in
\mathcal{B}(\Sigma; L; A)$, construct first a Banach space $B_{u_0}$
by considering the space of sections in $E := u_0^{-1} TW$ satisfying
the following boundary condition: Choose the unique collection of
loops $c$ that satisfy $\Phi\circ c =
\left.u_0\right|_{\partial\Sigma}$.  We can define a subbundle $F\le
\left.E\right|_{\partial\Sigma}$ over the boundary of the surface by
pushing $T_{c(z)}L$ with $D\phi$ into $E$.  We require the sections
$\sigma:\, \Sigma \to E$ to lie along the boundary $\partial\Sigma$ in
the subbundle $F$, and to be at the marked points $z_j \in
\partial\Sigma_j$ tangent to $A$.

Our aim will be to map these sections in a suitable way into
$\mathcal{B}(\Sigma; L; A)$.  For this, we first choose a Riemannian
metric on $L$ for which $A$ is totally geodesic.  Then we extend it to
a product metric on $\partial\Sigma \times L$.  There is an
$\epsilon_1 > 0$ such that $\phi$ restricted to any $\epsilon_1$--disk
centered at an $c(z)$ is an embedding.  Furthermore, we find an
$\epsilon_2>0$ such that $d\bigl(c(z),c(z')\bigr) < \epsilon_1/3$ for
any pair of points $z,z'\in\partial\Sigma$ such that $d(z,z') <
\epsilon_2$.  Let $\epsilon$ be smaller than $\min\{\epsilon_1/3,
\epsilon_2\}$, and let $U_\epsilon(c)$ be the $\epsilon$--neighborhood
of the loops $\bigl\{\bigl(z, c(z)\bigr)\bigr\}\subset \partial\Sigma
\times L$, i.e., the collection of all points $(z',x')$ that lie at
most at distance $\epsilon$ from the set of loops.  The restriction of
the immersion
\begin{equation*}
  \id\times \phi:\, \partial\Sigma\times L \looparrowright
  \Sigma\times W,\, (z,x) \mapsto \bigl(z,\phi(x)\bigr)
\end{equation*}
to $U_\epsilon(c)$ defines an embedded submanifold of $\Sigma\times
W$, because any two points $(z,x), (z',x')\in U_\epsilon(c)$ for which
$\bigl(z,\phi(x)\bigr) = \bigl(z',\phi(x')\bigr)$, obviously satisfy
$z = z'$, and as we will show $x$ and $x'$ both lie in an
$\epsilon_1$--disk around $c(z)$ such that $x = x'$.  Let
$\bigl(z_0,c(z_0)\bigr)$ be a point for which
$d\bigl(\bigl(z_0,c(z_0)\bigr), (z,x)\bigr) < \epsilon$, then using
the triangle inequality we get
\begin{equation*}
  d\bigl(\bigl(z,c(z)\bigr), (z,x)\bigr) \le
  d\bigl(\bigl(z_0,c(z_0)\bigr), (z,c(z))\bigr)
  + d\bigl(\bigl(z_0,c(z_0)\bigr), (z,x)\bigr) < \epsilon_1 \;,
\end{equation*}
which shows that $(z,x)$ lies closer than $\epsilon_1$ to $(z,c(x))$.

Now we can push the metric from $U_\epsilon(c)$ forward and extend it
to one on $\Sigma\times W$, so that $\bigl(\id\times
\phi\bigr)\bigl(U_\epsilon(c)\bigr)$ will be totally geodesic.

Let $\sigma\in B_{u_0}$ be one of the sections of $E$ described above.
If $\sigma$ is sufficiently small, then applying the geodesic
exponential map to the section $(0, \sigma)$ in $T(\Sigma\times E)$,
and then projecting to the $W$--component gives a map that lies in the
space $\mathcal{B}(\Sigma; L; A)$.  The construction described gives a
bijection between small sections and maps in $\mathcal{B}(\Sigma; L;
A)$ close to $u_0$.  The reason is that there is a smooth map that
allows us to regard any manifold in $M_1 \times M_2$ tangent to
$M_1\times \{x_2\}$ at $(x_1,x_2)$ as a graph over $M_1\times \{x_2\}$
in a neighborhood of that point.

Since we do not see locally the other intersection branches it follows
that the Cauchy Riemann equation defines a Fredholm operator on
$\mathcal{B}(\Sigma; L; A)$.  For a generic adapted almost complex
structure $J$, it follows that the moduli space
\begin{equation*}
  \solutions(\Sigma; L; A) = \bigl\{ u \in \mathcal{B}(\Sigma; L; A)
  \bigm|\,  \bar \partial_J u = 0 \bigr\}
\end{equation*}
is a smooth manifold.  In our case, we then have that $\moduli_\gamma
:= \solutions(\Disk; \GPS; \gamma) /G$ is a $1$--dimensional manifold.

\subsection{The boundary of the GPS}
\label{sec: boundary of GPS}

The standard definition of the plastikstufe requires the boundary
$\partial\plastikstufe{S}$ to be a regular leaf of the foliation
\cite{NiederkruegerPlastikstufe}.  That way, $\plastikstufe{S}
-S\times \{0\}$ is a totally real manifold, and gives thus a Fredholm
boundary condition for regarding holomorphic disks, at the same time
smooth holomorphic disks in the moduli space have to be transverse to
the foliation so that they cannot touch the boundary.

In our definition of the GPS, we want the contact form instead to
vanish on the boundary $S\times \partial \Disk$.  In this section, we
will show by an intersection argument that there is a neighborhood of
the boundary which blocks any holomorphic curve from entering it.  Our
definition thus implies at this point effectively the same statement
as the standard one.

\begin{proposition}
  \label{moser am rand der plastikstufe} Let $F$ be a maximally
  foliated submanifold inside a contact manifold $(M,\alpha)$.  Assume
  one of its boundary components to be diffeomorphic to $N\cong
  S\times\S^1$, with $S$ a closed manifold, such that the restriction
  $\left.\alpha\right|_{TF}$ of the contact form has the following
  properties on the collar neighborhood $N \times [0,\epsilon) =
  \{(s,e^{i\phi},r)\}$
  \begin{itemize}
  \item [(1)] $\left.\alpha\right|_{TF}$ vanishes on $N$ (in
    particular $N$ is a Legendrian submanifold), and
  \item [(2)] the interior of the collar is foliated and the leaves
    are $S\times\{e^{i\phi_0}\}\times(0,\epsilon)$, for any fixed
    $e^{i\phi_0}\in\S^1$.
  \end{itemize}

  Then there is a neighborhood of $N$ in $M$ that is contactomorphic
  to an open subset of
  \begin{equation*}
    \bigl(\R \times T^*S \times\S^1\times \R,\,
    dz + \lcan - r\,d\phi\bigr)
  \end{equation*}
  such that $N\times [0,\epsilon)$ lies in this model in $\{0\}\times
  S\times \S^1\times[0,\epsilon)$.
\end{proposition}
\begin{proof}
  First note that it is clear that the restriction of the contact form
  can be written on the collar neighborhood as
  \begin{equation*}
    \left.\alpha\right|_{TF} = f\,d\phi\;,
  \end{equation*}
  with a smooth function $f:\,N\times [0,\epsilon) \to \R_{\ge 0}$
  which only vanishes on $N\times\{0\}$.  The $2$--form $d\alpha$ is a
  symplectic form on the $(2n)$--dimensional kernel $\xi=\ker\alpha$,
  so in particular $\left.d\alpha\right|_{TF}$ cannot vanish on any
  point $p\in N$, because otherwise $T_pF$ would be an
  $(n+1)$--dimensional isotropic subspace of $(\xi_p,d\alpha)$.  It
  follows that $\partial_r f(p,0) > 0$, and so the map $\Phi:\,
  N\times[0,\epsilon) \to N\times\R,\, (p,r)\mapsto (p,f(p,r))$ is
  after a suitable restriction a diffeomorphism with inverse
  $\Phi^{-1}(p,r) = (p, f^{-1}_p(r))$, where we wrote $f_p(\cdot) :=
  f(p,\cdot)$.  The pull-back of $\left.\alpha\right|_{TF}$ under
  $\Phi^{-1}$ gives
  \begin{align*}
    \left(\Phi^{-1}\right)^*\left(f\,d\phi\right) &=
    f(p,f^{-1}_p(r))\,d\phi = r\,d\phi\;.
  \end{align*}
  Thus, we can assume after changing the orientation of $\S^1$ that
  $\alpha$ is of the form $-r\,d\phi$ on the collar neighborhood.

  Consider now the normal bundle of the submanifold $N\times
  [0,\epsilon)$ in $M$.  A trivialization can be obtained by realizing
  first that the Reeb field $\Reeb$ is transverse to $N$, because
  $\left.TF\right|_N$ lies in the contact structure, so that there is
  a small neighborhood, where $\Reeb$ is transverse to $F$.  Choose
  now an almost complex structure $J$ on $\xi = \ker\alpha$ that is
  compatible with $d\alpha$ such that $J$ leaves the space on $N$
  spanned by $\langle
  \partial_r,
  \partial_\phi\rangle$ invariant.  The submanifolds
  $S_{(e^{i\phi},r)} := S\times\{(e^{i\phi},r)\}$, with
  $(e^{i\phi},r)$ fixed, are all tangent to the contact structure, and
  it follows that $J\cdot TS_{(e^{i\phi},r)}$ is transverse to $F$,
  because if there was an $X\in TS$, such that $JX\in TF$, then
  \begin{align*}
    0< d\alpha(X,JX) = - dr\wedge d\phi (X,JX) = 0\;.
  \end{align*}
  With the tubular neighborhood theorem it follows that there is an
  open set around $N\times [0,\epsilon)$ diffeomorphic to $\R\times
  T^*S \times \S^1\times (-\epsilon,\epsilon)$, and the set
  $N\times[0,\epsilon)$ lies in $\{0\}\times S\times \S^1\times
  [0,\epsilon)$.
  
  In the final step, we use a version of the Moser trick explained for
  example in \cite[Theorem~2.24]{Geiges_Handbook} to find a vector
  field $X_t$ that isotopes the given contact form into the desired
  one $dz + \lcan - r\,d\phi$.  Let $\alpha_t$, $t\in[0,1]$, be the
  linear interpolation between both $1$--forms.  Assume there is an
  isotopy $\psi_t$ defined around $N$ such that $\psi_t^*\alpha_t =
  \alpha_0$.  The field $X_t$ generating this isotopy satisfies the
  equation
  \begin{equation*}
    \lie{X_t} \alpha_t +\dot \alpha_t = 0\;.
  \end{equation*}
  By writing $X_t = H_t\,R_t + Y_t$, where $H_t$ is a smooth function,
  $R_t$ is the Reeb vector field of $\alpha_t$, and $Y_t\in\ker
  \alpha_t$, we obtain plugging then $R_t$ into the equation above
  \begin{align*}
    dH_t(R_t) = -\dot \alpha_t(R_t)\;.
  \end{align*}
  The vector field $Y_t$ is completely determined by $H_t$, because
  $Y_t$ satisfies the equations
  \begin{align*}
    \iota_{Y_t}\alpha_t &= 0\;, \\
    \iota_{Y_t}d\alpha_t &= - dH_t - \dot \alpha_t\;,
  \end{align*}
  hence it suffices to find a suitable function $H_t$.  Consider the
  $1$-parameter family of Reeb fields $R_t$ as a single vector field
  on the manifold $[0,1]\times \bigl(\R\times
  T^*S\times\S^1\times\R\bigr)$.  Since $R_t$ is transverse to the
  submanifold $\widetilde N := [0,1]\times \bigl(\{0\}\times
  T^*S\times\S^1\times\R\bigr)$ along $[0,1]\times N\times
  [0,\epsilon)$, it is possible to define a solution $H_t$ to
  $dH_t(R_t) = -\dot \alpha_t(R_t)$, such that
  $\left.H_t\right|_{\widetilde N}\equiv 0$.  In fact, because
  $\left.\dot\alpha\right|_{N\times[0,\epsilon)} = 0$, it follows that
  $\left.dH_t\right|_{N\times[0,\epsilon)} = 0$, and so the vector
  field $X_t = H_t\,R_t + Y_t$ vanishes on $N\times[0,\epsilon$.
  Hence $X_t$ can be integrated on a small neighborhood of the collar
  $N\times[0,\epsilon)$, and $N\times[0,\epsilon)$ is not moved under
  the flow, which finishes the proof of the proposition.
\end{proof}

We can easily choose a compatible adapted almost complex structure $J$
on the symplectization
\begin{equation*}
  \Bigl(W=\R\times (\R\times T^*S\times \S^1\times\R),\, \omega =
  d\bigl(e^t\,(dz + \lcan - r\,d\phi)\bigr)\Bigr)\;,
\end{equation*}
with coordinates $\{(t,z;\POS,\MOM;e^{i\phi},r)\}$.  Observe that the
Reeb field is given by $\Reeb = e^{-t}
\partial_z$, and that the kernel of $\alpha$ is spanned by the vectors
$X-\lcan(X)\,\partial_z$ for all $X\in T(T^*S)$, $\partial_\phi +
r\,\partial_z$ and $\partial_r$.  Choose a metric $g$ on $S$, and let
$J_0$ be the $d\lcan$--compatible almost complex structure on $T^*S$
constructed in~\cite[Appendix~B]{NiederkruegerPlastikstufe}.

With this, we can define a $J$ on $W$ by setting $J\partial_t =
\Reeb$, $J\Reeb = -\partial_t$, $J\,\partial_r = -\partial_\phi -
r\,\partial_z$, $J(\partial_\phi + r\,\partial_z) =
\partial_r$, and $J (X -\lcan(X)\,\partial_z) = J_0 X -
\lcan(J_0X)\,\partial_z$.  The last two equations can also be written
as $J\,\partial_\phi = \partial_r + re^t\,\partial_t$ and $JX = J_0 X
- e^t\,\lcan(X)\,\partial_t - \lcan(J_0X)\,\partial_z$.

As a matrix, the complex structure $J$ takes the form
\begin{align*}
  J(t;z;\POS,\MOM;e^{i\phi},r) &=
    \begin{pmatrix}
      0  & -e^t & - e^t\,\lcan & re^t & 0 \\
      e^{-t} & 0 & -\lcan\circ J_0 & 0 & -r \\
      0 & 0 & J_0 & 0 & 0 \\
      0 & 0 & 0 & 0 & -1 \\
      0 & 0 & 0 & 1 & 0
    \end{pmatrix}
    \;.
\end{align*}
Note that the center row and column represent linear maps from or to
$T(T^*S)$.  A lengthy computation shows that this structure is
compatible with $\omega$.

\begin{proposition}
  \label{fast komplexe struktur beim rand der plastikstufe} The almost
  complex manifold $(W,J)$ can be mapped with a biholomorphism to
  \begin{align*}
    (\C\times T^*S\times\C^*, i\oplus J_0\oplus i)\;.
  \end{align*}
  In this model, the contact manifold $M$ is described by the set
  \begin{align*}
    M &\cong \Bigl\{(x + iy;\POS,\MOM; w) \in \C\times T^*S\times\C^*
    \Bigm| x = -\frac{2+\norm{\MOM}^2 + (\ln\abs{w})^2}{2}\Bigr\}\;,
    \intertext{and the maximally foliated submanifold $F$ is} F &\cong
    \Bigl\{(x ;\POS,0; w) \in \R\times S\times\C^* \Bigm| x = -\frac{2
      + (\ln\abs{w})^2}{2},\; \abs{w}\ge 1\Bigr\} \subset \C\times
    T^*S\times\C^*\;.
  \end{align*}
\end{proposition}
\begin{proof}
  The desired biholomorphism is
  \begin{align*}
    \Phi(t,z;\POS,\MOM;e^{i\phi},r) &= (\tilde t,\tilde z;
    \tilde\POS,\tilde\MOM; \tilde r e^{i\tilde\phi}) = \left(-e^{-t} -
      F -\frac{r^2}{2},z; \POS,\MOM; e^re^{i\phi}\right)\;,
  \end{align*}
  with the function
  \begin{equation*}
    F:\, T^*M \to \R,\,\, (\POS,\MOM) \mapsto
    \frac{\norm{\MOM}^2}{2}\;.
  \end{equation*}
  It brings $J$ into standard form with respect to the coordinate
  pairs $(\tilde r e^{i\tilde\phi})$, $(\tilde t,\tilde z)$.  More
  explicitely, by pulling back $J$ under the diffeomorphism
  \begin{align*}
    \Phi^{-1}(\tilde t,\tilde z;\tilde\POS,\tilde\MOM;\tilde
    re^{i\tilde \phi}) &= (t,z; \POS,\MOM; e^{i\phi},r) =
    \left(-\ln(-\tilde t- F - \frac{(\ln\tilde r)^2}{2}),\tilde
      z;\tilde\POS,\tilde\MOM; e^{i\tilde \phi},\ln\tilde r \right)\;,
  \end{align*}
  i.e.\ by computing $D\Phi\cdot J\cdot D\Phi^{-1}$, we obtain the
  matrix
  \begin{align*}
    D\Phi\cdot J\cdot D\Phi^{-1} &=
    \begin{pmatrix}
      0 & -1 & -\lcan - dF\circ J_0 & 0 & 0\\
      1 & 0 &  dF -\lcan\circ J_0 & 0 & 0\\
      0 & 0 & J_0 & 0 & 0 \\
      0 & 0 & 0 & 0 & - 1/\tilde r \\
      0 & 0 & 0 & \tilde r & 0
    \end{pmatrix}\;,
  \end{align*}
  and since, according to
  \cite[Appendix~B]{NiederkruegerPlastikstufe}, $dF\circ J_0 =
  -\lcan$, this gives the desired normal form.
\end{proof}

\begin{proposition}
  Let $F$ be a maximally foliated submanifold in a contact manifold
  $(M,\alpha)$.  Let $(W,\omega)$ be a symplectic filling of $M$, and
  assume $F$ to have a boundary component of the type explained in
  Proposition~\ref{moser am rand der plastikstufe}.  There is a
  neighborhood $U$ of the boundary with an almost complex structure,
  which prevents any holomorphic curve $u:\,(\Sigma,\partial \Sigma)
  \to (W,F)$ that has in $F$ contractible boundary components, from
  entering~$U$.
\end{proposition}
\begin{proof}
  Choose for the neighborhood $U$ of $\partial F$ the model described
  in Proposition~\ref{fast komplexe struktur beim rand der
    plastikstufe} together with the almost complex structure $J$ given
  there.  This $J$ can be easily extended over the whole filling
  $(W,\omega)$.  Note that the neighborhood is foliated by
  $J$-holomorphic codimension~$2$ manifolds of the form $N_C :=
  \{C\}\times T^*S\times\C^*$ for any fixed complex number $C$.

  Let now $u:\, (\Sigma,\partial\Sigma) \to (W,F)$ be a holomorphic
  curve that has in $F$ contractible boundary, and assume that $u$
  intersects the model neighborhood $U$.  Write the restriction of $u$
  to $V:= u^{-1}(U) \subset\Sigma$ as
  \begin{equation*}
    \left.u\right|_V:\, V \to \C\times T^*S\times\C^*,\,
    z \mapsto \bigl(u_1(z);\POS(z),\MOM(z);u_2(z)\bigr)   \;.
  \end{equation*}
  First, we will show that the imaginary part of the first coordinate
  $u_1$ is constant.  If it was not, then there would be by Sard's
  Theorem a regular value $c_y\in\R$, such that $u_1^{-1}(\R+ic_y)$
  consists of finitely many regular $1$--dimensional submanifolds.
  The real part of $u_1$ changes along these submanifolds, because
  $u_1$ satisfies the Cauchy-Riemann equation.  Hence it is possible
  to find a complex number $c_x+ic_y\in\C$ such that $N_{c_x+ic_y}$
  has finitely many transverse intersection points with $u$.  By our
  assumption, it is possible to cap off the holomorphic curve $u$ by
  adding disks that lie inside~$F$.  Note that $N_{c_x+ic_y}$ is the
  boundary of the submanifold
  \begin{equation*}
    \widetilde N_{c_x+ic_y} := \bigl\{x+ic_y\bigm|\, x\in[c_x,\infty)\bigr\}
    \times T^*S\times\C^*\;.
  \end{equation*}
  The intersection of $\widetilde N_{c_x+ic_y}$ with $M$ gives a
  submanifold that is disjoint from $F$, and this submanifold together
  with $N_{c_x+ic_y}$ represents the trivial homology class in
  $H_{2n-2}(W)$.

  The only intersections between the capped off holomorphic curve and
  $\partial (W\cap \widetilde N_{c_x+ic_y})$ lie in the subsets, where
  both classes are represented by $J$-holomorphic manifolds.  Hence
  the intersection number is positive, but since $\partial (W\cap
  \widetilde N_{c_x+ic_y})$ represents the trivial class in homology,
  this is clearly a contradiction.

  It follows that $\ImaginaryPart u_1$ is constant, and with the
  Cauchy-Riemann equation, we immediately obtain that the real part of
  $u_1$ must also be constant.  This in turn means that the
  holomorphic curve is completely contained in the neighborhood,
  because the only way that $u$ could not be completely contained in
  the neighborhood $U$ is if $\abs{u_2(z)}$ changes sufficiently or if
  $\MOM$ grows, but in both cases $u$ will hit the hypersurface $M$
  before leaving $U$.  Hence $u$ is contained in $U$.  Consider now
  the $T^*S$--part of $u$.  Since $u$ sits on $F$, it follows that the
  $T^*S$--part has boundary on the zero-section of $T^*S$, and so it
  has no energy, and is thus constant.  So far, it follows that $u(z)$
  can be written as $u(z) = (c_x+ic_y;\POS_0,0; f(z))$, where
  $f:\,\Disk \to \C^*$ such that $f(\S^1)$ lies in one of the circles
  of fixed radius $R$ or~$1/R$.  But in fact, only the circle of
  radius $R>1$ lies in $F$, hence all boundary component of $\Sigma$
  are mapped to the circle of radius $R$, and so by the maximum
  principle $\abs{f}^2$ is bounded by $R^2$, and by the boundary point
  lemma the derivative of $f$ along the boundary may nowhere vanish.
\end{proof}

\subsection{Bubbling analysis}\label{sec: bubbling}

To obtain compactness of our moduli space, we need to distinguish two
cases: Either the first derivatives of the sequence are uniformly
bounded from the beginning, and we have subsequence with a clean limit
(after adapting the standard result to the immersed boundary
condition), or if the first derivatives explode, we show that we do
find a global uniform bound on the derivatives if we reparametrize the
disks in a suitable way.

\begin{theorem}\label{theorem: convergence for bounded maps}
  Let $\Sigma$ be a Riemann surface that does not need to be compact,
  and may or may not have boundary.  Let $\Omega_k\subset \Sigma$ be a
  family of increasing open sets that exhaust $\Sigma$, i.e.,
  \begin{equation*}
    \cup_k \Omega_k = \Sigma  \text{ and } \Omega_k\subset \Omega_{k'}
    \text{ for $k\le k'$.}
  \end{equation*}
  Define $\partial \Omega_k := \Omega_k \cap \partial \Sigma$.  Let
  $(W,J)$ be a compact almost complex manifold that contains a totally
  real immersion $\phi:\, L \looparrowright W$ of a compact manifold
  $L$.

  Let $u_k$ be a sequence of holomorphic maps $u_k:\, \bigl(\Omega_k,
  \partial\Omega_k\bigr) \to \bigl(W, \phi(L)\bigr)$ whose derivatives
  are uniformly bounded on compact sets, i.e., if $K\subset \Sigma$ is
  a compact set, then there exists a constant $C(K) > 0$ such that
  \begin{equation*}
    \norm{Du_k(z)} \le C(K)
  \end{equation*}
  for all $k$ and all $z\in\Omega_k\cap K$.  Additionally assume that
  the restriction of $u_k$ to the boundary $\partial\Omega_k$ lifts to
  a collection of smooth paths $u^L_k:\, \partial\Omega_k \to L$ such
  that $\phi\circ u^L_k = \left.u_k\right|_{\partial \Omega_k}$.

  Then there exists a subsequence of $u_k$ that converges on any
  compact subset uniformly with all derivatives to a holomorphic curve
  $u_\infty:\, (\Sigma, \partial\Sigma) \to \bigl(W,\phi(L)\bigr)$,
  whose boundary lifts to a collection of smooth paths $u^L_\infty:\,
  \partial\Sigma \to L$, and the boundary paths $u_k^L$ also converge
  locally uniformly to $u^L_\infty$.
\end{theorem}
\begin{proof}
  The theorem is well-known in case that $\partial\Sigma = \emptyset$
  or that $\phi(L)$ is an embedded totally real submanifold (see for
  example \cite[Theorem~4.1.1]{McDuffSalamonJHolo}).  In fact, our
  situation can be reduced to one, where we can apply this standard
  result.  Using Arzelà-Ascoli it is easy to find a subsequence $u_k$
  that converges uniformly in $C^0$ on any compact set to a continuous
  map $u_\infty$, and such that the lifts $u^L_k:\, \partial\Omega_k
  \to L$ converge in $C^0$ on any compact set to a lift $u^L_\infty:\,
  \partial \Sigma \to L$ with $\phi\circ u^L_\infty =
  \left.u_\infty\right|_{\partial\Sigma}$.

  Let $K \subset \Sigma$ be a compact set on which we want to show
  uniform $C^\infty$--convergence.  If $\partial K:= K\cap
  \partial \Sigma$ is empty, then the uniform converges for the
  derivatives follows from the standard result.  If $\partial K$ is
  non-empty, then cover $u^L_\infty(\partial K)$ with a finite
  collection of open sets $V_1,\dotsc,V_N$ on each of which $\phi$ is
  injective.  We can choose smaller open subsets $V_j' \subset V_j$
  whose closure $\overline{V_j'}$ is also contained in $V_j$, and
  whose union $V_1'\cup \dotsm \cup V_N'$ still cover
  $u^L_\infty(\partial K)$.

  Cover also $K$ itself with open sets $U_k$ that either do not
  intersect the boundary $\partial K$ or if $U_k\cap \partial K \ne
  \emptyset$, then there is a $V_j'$ such that $u^L_\infty(U_k\cap
  \partial K)\subset V_j'$.  Only finitely many $U_k$ are needed to
  cover $K$.  We get for every $U_k\cap K$ uniform
  $C^\infty$--convergence, because if $U_k$ intersects now $\partial
  K$ we can use the standard result: For $n$ large enough $u_n(U_k\cap
  \partial K\cap \Omega_n)$ will be contained in the larger subset
  $V_j$, on which $\phi$ is an embedding.
\end{proof}

\begin{theorem}[Gromov compactness]
  Let
  \begin{equation*}
    u_n:\, (\Disk,\partial \Disk) \to (W,\GPS)
  \end{equation*}
  be a sequence of holomorphic disks that represent elements in the
  moduli space $\moduli_\gamma$.

  There exists a family $\phi_n:\,\Disk\to \Disk$ of biholomorphisms
  such that $u_n\circ \phi_n$ contains a subsequence converging
  uniformly in $C^\infty$ to a holomorphic disk
 \begin{equation*}
    u_\infty:\,(\Disk,\partial \Disk) \to (W,\GPS)
 \end{equation*}
  that represents again an element in $\moduli_\gamma$.
\end{theorem}

\begin{proof}
  Choose an arbitrary $J$--compatible metric on $W$, and endow the
  disk $\Disk\subset\C$ with the standard metric $g_0$ on the complex
  plane.  Denote by $\norm{Du_k(z)}_{\Disk}$ the norm of the
  differential of $u_k$ at a point $z\in\Disk$ with respect to $g_0$
  on the disk, and the chosen metric on $W$.  If $\norm{Du_k}_{\Disk}$
  is uniformly bounded for all $k\in\N$ and all $z\in \Disk$, then by
  Theorem~\ref{theorem: convergence for bounded maps} we are done.

  So assume this to be false, then there exists (by going to a
  subsequence if necessary) a sequence $z_k\in\Disk$ such that
  \begin{equation*}
    \norm{Du_k(z_k)}_{\Disk} \to \infty \;,
  \end{equation*}
  and in fact by using rotations, we may assume that all $z_k$ lie on
  the interval $[0,1]$.

  Let $\HyperbolicPlane \subset \C$ be the upper half plane $\{z|\,
  \ImaginaryPart z \ge 0\}$ endowed with the standard metric, and
  denote by $\norm{Dv(z)}_{\HyperbolicPlane}$ the norm of the
  differential of a map $v:\,\HyperbolicPlane \to W$ at a point
  $z\in\HyperbolicPlane$ with respect to the standard metric on the
  half plane, and the chosen metric on $W$.  We can map the half plane
  into the unit disk using the biholomorphism
  \begin{equation*}
    \Phi: \, \HyperbolicPlane \to \Disk -\{-1\}, \,
    z\mapsto \frac{i - z}{i+z} \;,
  \end{equation*}
  and use this to pull-back the sequence of disks to
  $u_k^{\HyperbolicPlane} := u_k\circ \Phi:\, (\HyperbolicPlane,\R)
  \to (W,\GPS)$.  The map $\Phi$ is not an isometry, but on compact
  sets of the upper half plane, $\Phi^*g_0$ is equivalent to the
  standard metric.  Hence it follows that also $\norm{D
    u_k^{\HyperbolicPlane}}_{\HyperbolicPlane}$ cannot be bounded on
  the segment $I := \{it|\, t\in[0,1]\}$.  Let $x_k\in I$ be a point
  where $\norm{D u_k^{\HyperbolicPlane}}_{\HyperbolicPlane}$ takes its
  maximum on $I$.

  Apply the Hofer Lemma (see for example
  \cite[Lemma~4.6.4]{McDuffSalamonJHolo}) for fixed $k$, and $\delta =
  1/2$, that means, restrict $\norm{D
    u_k^{\HyperbolicPlane}}_{\HyperbolicPlane}$ to the unit disk
  $\Disk(x_k) \cap \HyperbolicPlane$.  There is a positive
  $\epsilon_k$ with $\epsilon_k \le 1/2$, and a $y_k \in
  \Disk[1/2](x_k)\cap \HyperbolicPlane$ such that
  \begin{align*}
    \norm{D u_k^{\HyperbolicPlane}(x_k)}_{\HyperbolicPlane} &\le
    2\epsilon_k\, \norm{D
      u_k^{\HyperbolicPlane}(y_k)}_{\HyperbolicPlane} \intertext{and}
    \norm{D u_k^{\HyperbolicPlane}(z)}_{\HyperbolicPlane} &\le
    2\,\norm{D u_k^{\HyperbolicPlane}(y_k)}_{\HyperbolicPlane}
  \end{align*}
  for all $z\in\Disk[\epsilon_k](y_k)\cap \HyperbolicPlane$.

  Set $c_k := \norm{D u_k^{\HyperbolicPlane}
    (y_k)}_{\HyperbolicPlane}$.  First we will show that for large
  $k$, all the disks $\Disk_{\epsilon_k}(y_k)$ intersect the boundary
  $\partial\HyperbolicPlane = \R$ of the half plane.  Even stricter,
  there exists a constant $K>0$ such that $c_k\,\ImaginaryPart(y_k) <
  K$ for all $k$ (if the disks intersect the real line, we have
  $\ImaginaryPart y_k < \epsilon_k$, multiplying with $c_k$ on both
  sides would still allow the left side to be unbounded).  Suppose
  that such a constant did not exist, so that by going to a
  subsequence, $c_k \ImaginaryPart y_k$ converges monotonously to
  $\infty$.  Define $H_k:= \{z\in\C|\,\ImaginaryPart z \ge
  -c_k\ImaginaryPart y_k\}$, and a sequence of biholomorphisms
  \begin{equation*}
    \phi_k:\, \Disk_{\epsilon_kc_k} \cap H_k \to \Disk[\epsilon_k](y_k)
    \cap \HyperbolicPlane, \, z \mapsto y_k + \frac{z}{c_k}\;.
  \end{equation*}
  Pulling back, we find holomorphic maps $\widehat u_k :=
  u_k^{\HyperbolicPlane}\circ\phi_k:\, \Disk_{\epsilon_kc_k} \cap H_k
  \to W$ with $\norm{D\widehat u_k(0)} = 1$, and $\norm{D\widehat u_k}
  \le 2$ everywhere else.  Using Theorem~\ref{theorem: convergence for
    bounded maps} (or just for example
  \cite[Theorem~4.1.1]{McDuffSalamonJHolo}), proves that there exists
  a subsequence that converges locally uniformly with all derivatives
  to a non-constant map $\widehat u_\infty:\,\C\to W$.  The standard
  removal of singularity theorem yields then a non-constant
  holomorphic sphere, which cannot exist in an exact symplectic
  manifold.  Thus there is a constant $K>0$ such that
  $c_k\,\ImaginaryPart(y_k) < K$.

  Now we slightly modify the charts used above to keep the boundary of
  the reparametrized domains on the height of the real line.  Set
  $y_k':= c_k\ImaginaryPart y_k$ and $r_k:= \epsilon_k c_k$, and
  consider the following sequence of biholomorphisms
  \begin{equation*}
    \psi_k:\, \Disk_{r_k}(iy_k') \cap \HyperbolicPlane
    \to \Disk[\epsilon_k](y_k) \cap \HyperbolicPlane, \,
    z \mapsto \frac{z}{c_k} + \RealPart y_k \;.
  \end{equation*}
  Note that the intersection of $\Disk[\epsilon_k](y_k)$ with the real
  line is given by the interval
  \begin{equation*}
    \Disk[\epsilon_k](y_k) \cap \R = \Disk(x_k) \cap \R \subset (-1,1) \;.
  \end{equation*}
  The image of the interval $(-1,1)$ under $\Phi$ is the segment on
  the boundary of the unit disk enclosed between the angles $(-\pi/2,
  \pi/2)$.  This means that the boundary part of the disk that is
  affected by the reparametrization lies on the right half of the
  complex plane.

  On the domain of the reparametrized maps $\widehat u_k :=
  u_k^{\HyperbolicPlane}\circ\psi_k:\, \Disk_{r_k}(i y_k') \cap
  \HyperbolicPlane \to W$ we have $\norm{D\widehat u_k} \le 2$, and
  $\norm{D\widehat u_k(i y_k')} = 1$.  We can also find a subsequence
  of $\widehat u_k$ with increasing domains, i.e., $\Disk_{r_k}(i
  y_k') \subset \Disk_{r_l}(i y_l')$ for all $l\ge k$, by using that
  the $y_k'$ are all bounded while the radii of the disks $r_k$ become
  arbitrarily large.  Then Theorem~\ref{theorem: convergence for
    bounded maps} provides a subsequence of the $\widehat u_k$ that
  converges locally uniformly with all derivatives to a holomorphic
  map $\widehat u_\infty:\,(\HyperbolicPlane,\R)\to (W,\GPS)$.  To see
  that $\widehat u_\infty$ is not constant, take a subsequence such
  that $y_k'$ converges to $y_\infty'$.  The norm of the derivative of
  $u_\infty$ at $iy_\infty$ is $\norm{D \widehat u_\infty(i
    y_\infty')} = 1$, because $\norm{D \widehat u_\infty (i y_\infty')
    - D\widehat u_k(i y_k')} \le \norm{D\widehat u_\infty(i y_\infty')
    - D\widehat u_\infty (i y_k')} + \norm{D\widehat u_\infty(i y_k')
    - D\widehat u_k (i y_k')}$ becomes arbitrarily small.  The first
  term is small, because the differential of $\widehat u_\infty$ is
  continuous, the second can be estimated by using that the
  convergence of $\widehat u_k$ to $\widehat u_\infty$ is uniform on a
  small compact neighborhood of $i y_\infty'$.

  Let us come back to the initial family of disks $u_k:\,
  (\Disk,\partial \Disk) \to (W,\GPS)$.  The maps $\psi_k$ induce
  reparametrizations of the whole disk by $\Phi\circ \psi_k\circ
  \Phi^{-1}$.  The image of a compact subset of $\Disk - \{-1\}$ under
  $\Phi^{-1}$ is a compact subset in $\HyperbolicPlane$, so that we
  get on any compact subset of $\Disk - \{-1\}$ uniform
  $C^\infty$--convergence of $u_k\circ \psi_k$ to $u_\infty :=
  \widehat u_\infty\circ \Phi^{-1}$.  To complete the proof of our
  compactness theorem, we have to show that the first derivatives of
  $u_k\circ \psi_k$ are also uniformly bounded in a neighborhood of
  $\{-1\}$.

  Rotate the disk $\Disk$ by multiplying its points by $e^{i\pi}= -1$
  such that $-1$ lies at $1$.  Then the holomorphic curve
  \begin{equation*}
    \bigl(\HyperbolicPlane-\{0\}, (-\infty,0)\cup (0,\infty)\bigr) \to
    (W,\GPS), z \mapsto u_\infty\bigl(-\Phi(z)\bigr)
  \end{equation*}
  has finite energy, and we can apply the removal of singularity
  theorem in the form described in Theorem~\ref{theorem: removal
    singularity}.  The consequence for $u_\infty$ is that the
  composition $\left.\theta\circ u_\infty\right|_{\partial \Disk
    -\{-1\}}$ extends to a continuous map $\S^1 \to \S^1$ that is
  strictly monotonous.  In fact, $\left.\theta\circ
    u_\infty\right|_{\partial \Disk -\{-1\}}$ covers the whole circle
  with exception of the point
  \begin{equation*}
    e^{i\,\phi_\infty} := \lim_{e^{i\phi}\to -1}
    \theta\circ u_\infty(e^{i\phi})\;,
  \end{equation*}
  and so for any $\epsilon$--neighborhood $U_\epsilon \subset \S^1$ of
  $e^{i\,\phi_\infty}$, we find a $\delta>0$ such that $\{\theta\circ
  u_k\circ\psi_k(e^{i\phi})|\,\phi \in (-\pi/2 + \delta, \pi/2 -
  \delta) \}$ covers for any sufficiently large $k$ the complement
  $\S^1 - U_\epsilon$ of $U_\epsilon$.  Let $K$ be the segment
  $\{e^{i\phi}|\, \phi\in(-\pi/2,\pi/2)\}$.  Remember that the images
  of $K$ exhausts $\partial\Disk - \{-1\}$ if we apply the
  reparametrizations $\psi_k$, and so it follows in particular that
  the unparametrized disks $u_k:\, (\Disk,\partial \Disk) \to
  (W,\GPS)$ intersect on $K$ for sufficiently large $k$ almost all
  leaves of the foliation of the GPS.

  Assume now that the first derivatives of the $u_k\circ \psi_k$ are
  not uniformly bounded in a neighborhood of $\{-1\}$.  By the same
  reasoning, it follows that the $u_k\circ \psi_k$ intersect almost
  all leaves of the GPS on the segment $K' = \{e^{i\phi}|\,
  \phi\in(\pi/2,\pi/3)\}$, but this yields a contradiction.
\end{proof}

In our special situation, we only need the following very weak form of
removal of singularity, which states that a holomorphic curve that has
a puncture on its boundary approaches the same leaf of the foliation
from both sides of the puncture.

\begin{theorem}[Removal of singularity]\label{theorem: removal
    singularity}
  Let $(W, \omega)$ be a compact manifold with exact symplectic form
  $\omega = d\alpha$, and with convex boundary $\partial W =
  (M,\alpha)$.  Assume that $M$ contains a GPS $\phi:\, S\times \Disk
  \looparrowright M$, and choose an adapted almost complex structure
  $J$ on $W$.  Assume
  \begin{equation*}
    u:\, \Bigl(\Disk[\epsilon] \cap \HyperbolicPlane - \{0\},
    (-\epsilon,0) \cup (0,\epsilon)\Bigr) \to \bigl(W,\GPS\bigr)
  \end{equation*}
  to be a non-constant holomorphic curve that has finite energy.
  Recall that there was a continuous map $\theta:\, \GPS^* \to \S^1$
  that labels the leaves of the foliation on the GPS.

  We find a continuous path $\widehat c:\, (-\epsilon, \epsilon) \to
  \S^1$ with
  \begin{equation*}
    \left.\widehat c\right|_{(-\epsilon,0) \cup (0,\epsilon)} =
    \left.\theta\circ u\right|_{(-\epsilon,0) \cup (0,\epsilon)}\;.
  \end{equation*}
  A more geometric way to state this result is to say that the
  boundary segments of the holomorphic curve approach from both sides
  of $0$ asymptotically the same leaf.
\end{theorem}
\begin{proof}
  One of the basic ingredients in all proofs of this type is the
  following estimate for the energy of $u$
  \begin{equation*}
    E(u) = \int_{\Disk[\epsilon]\cap\HyperbolicPlane-\{0\}} u^*\omega
    = \int_0^\epsilon \int_{\gamma_r} \frac{\abs{\partial_\phi u}^2}{r^2}\,
    r\, dr\wedge d\phi \\
    \ge  \int_0^\epsilon \left(\int_{\gamma_r} \abs{\partial_\phi u}\, d\phi\right)^2
    \, \frac{dr}{2\pi r} \\
    = \int_0^\epsilon  \frac{L(\gamma_r)^2}{2\pi r} \, dr \;,
  \end{equation*}
  where $\gamma_r$ is the image $\bigl\{u(re^{i\phi})\bigm|\,
  \phi\in[0,\pi]\bigr\}$ of the half-circle of radius $r$ in the
  hyperbolic plane, and $L(\gamma_r)$ is its length with respect to
  the compatible metric on $W$.  It is clear that $L(\gamma_r)$ cannot
  be bounded from below, because the energy $E(u)$ is finite.

  Denote the segments composing the map $\left.\theta\circ
    u\right|_{(-\epsilon,0) \cup (0,\epsilon)}$ by $c_-:\,
  (-\epsilon,0) \to \S^1$, and $c_+:\, (0,\epsilon) \to \S^1$.  By
  Corollary~\ref{kurven transvers zu rand}, both maps $c_\pm$ are
  strictly increasing.

  It easily follows that the $c_\pm$ are bounded close to $0$ (in the
  sense that they do not turn infinitely often as $z\to\infty$),
  because there is a sequence of radii $r_k$ with $r_k \to 0$ such
  that $L(\gamma_{r_k})\to 0$.  Denote $(\Disk[r_1] - \Disk[r_k])\cap
  \HyperbolicPlane$ by $D(r_1,r_k)$.  Then
  \begin{equation*}
    E\bigl(\left.u\right|_{D(r_1,r_k)}\bigr) =
    \int_{\partial D(r_1,r_k)} u^*\alpha
    \ge \int_{\text{\makebox[1.5cm][l]{$[-r_1,-r_k]\cup [r_k,r_1]$}}} u^*\alpha
    - \bigl(L(\gamma_{r_1}) + L(\gamma_{r_k})\bigr)\, \max \norm{\alpha}
    \to \infty\;.
  \end{equation*}
  It follows that we find continuous extensions $\widehat c_-:\,
  (-\epsilon,0] \to \S^1$, and $\widehat c_+:\, [0,\epsilon) \to
  \S^1$.  If $\widehat c_-(0) = \widehat c_+(0)$, we are done, so
  assume these limits to be different.  Choose a small $\delta > 0$,
  such that the $\delta$--neighborhoods $U_-, U_+\subset\S^1$ around
  $\widehat c_-(0)$ and $\widehat c_+(0)$ respectively do not overlap.
  There is an $\epsilon'>0$ for which the segment $[0,\epsilon')$ is
  contained in $\widehat c_+^{-1}(U_+)$, and $(-\epsilon',0]$ is
  contained in $\widehat c_-^{-1}(U_-)$, and all the points in
  $u\bigl((0,\epsilon')\bigr)$ are at distance more than $C>0$ from
  the points in $u\bigl((-\epsilon',0)\bigr)$.  In particular it
  follows that the length $L(\gamma_r)$ for any $r\in(0,\epsilon')$ is
  larger than $C$, and so by the energy inequality at the beginning of
  the proof, we get a contradiction to $\widehat c_-(0) \ne \widehat
  c_+(0)$.
\end{proof}

\bibliographystyle{amsalpha}

\section{Outlook and open questions}\label{sec: open questions}

One obvious application of the observations made in this paper is the
definition of a capacity invariant for contact manifolds.
Unfortunately, we were not able to measure the ``size'' efficiently in
a numerical way so that our invariant is rather rough.

To measure the capacity, we choose a contact manifold $(N,\xi_N)$ that
will serve as the ``testing probe''.

Then we can define for any contact manifold $(M,\xi)$ with $\dim M =
2k + \dim N$, and $k\ge 1$, an invariant $C_{\xi_N}$ defined as
follows
\begin{equation*} C_{\xi_N}(M,\xi) =
  \begin{cases}
    0 &\quad \text{$(N,\xi_N)$ cannot be embedded with trivial normal bundle into $M$;} \\
    \infty & \text{$N\times \R^{2k}$ with the standard contact form
      can be embedded into $M$;} \\
    1 & \text{otherwise, that means $(N,\xi_N)$ can be embedded with
      trivial normal}\\ &\text{ bundle into $M$, but not with the full
      neighborhood.}
  \end{cases}
\end{equation*}

This way, we obtain for the standard sphere $\bigl(\S^{2n-1},
\xi_0\bigr)$ that $C_{\xi_0}(M,\xi) = \infty$ for any contact manifold
$(M,\xi)$.  If $\bigl(N,\xi_-\bigr)$ is an overtwisted contact
$3$--manifold, and if $(M,\xi)$ is a manifold with exact symplectic
filling, then $C_{\xi_-}(M,\xi) < \infty$.

The most important problem in this context would be to find examples
of contact manifolds that do allow the embedding of an overtwisted
contact manifold $N$ with the full model neighborhood
$N\times\R^{2k}$, because otherwise it is so far unclear whether the
capacity $C_{\xi_N}$ is able to distinguish any manifolds.  Possible
candidates to check are the following:

\begin{question}
  Let $(M,\alpha)$ be a closed contact manifold.  Bourgeois described
  in \cite{BourgeoisTori} a construction of a contact structure on
  $M\times \T^2$ for which every fiber $M\times\{p\}$ with $p\in \T^2$
  is contactomorphic to the initial manifold.  How large is the
  tubular neighborhood of such a fiber?
\end{question}

\begin{question}
  Giroux conjectures that contact manifolds of arbitrary dimension
  obtained from the negative stabilization of an open book should be
  ``overtwisted''.  The simplest example of such a manifold is a
  sphere $(\S^{2n-1}, \alpha_-)$ constructed by taking the cotangent
  bundle $T^*\S^{n-1}$ for the pages, and a negative Dehn-Seidel twist
  as the monodromy map (see Example~\ref{example: branched cover}).
  How large is the tubular neighborhood of $(\S^3,\alpha_-)$ in a
  higher dimensional sphere?
\end{question}

\bibliography{main}


\end{document}